\documentclass[a4paper,reqno]{amsart}  
\usepackage{graphicx}
\usepackage[margin=3.4cm]{geometry}
\usepackage[english]{babel}
\usepackage{amsfonts, color}
\usepackage{amstext, amssymb, amsmath}
\usepackage{amssymb}
\usepackage{algorithm, algorithmic}
\usepackage{soul}
\usepackage{mathrsfs}

\newcommand{\R}{\mathbb{R}}

\theoremstyle{plain}
\begingroup
\newtheorem{theorem}{Theorem}[section]
\newtheorem{lemma}[theorem]{Lemma}
\newtheorem{proposition}[theorem]{Proposition}

\endgroup

\theoremstyle{definition}
\begingroup
\newtheorem{defn*}[theorem]{Definition}
\newtheorem{remark}[theorem]{Remark}
\newtheorem{example}[theorem]{Example}
\endgroup

\begin{document}


\title[A piecewise conservative method for   unconstrained convex optimization]
{A piecewise conservative method for unconstrained convex optimization}

\author[ A. Scagliotti and P. Colli Franzone]{Alessandro Scagliotti and  Piero Colli Franzone}


\address[A.~Scagliotti]{Scuola Internazionale Superiore di Studi Avanzati, Trieste, Italy}
\email{ascaglio@sissa.it}
\address[P.~Colli~Franzone]{Dipartimento di Matematica, Universit\`a di Pavia, Italy}
\email{piero.collifranzone@unipv.it }

\begin{abstract}
We consider a continuous-time optimization method
based on a dynamical system, where 
a massive particle starting at rest moves in the 
conservative force field generated by the objective 
function, without any kind of friction.
We formulate a restart criterion 
based on the mean dissipation of the kinetic
energy, and we prove a global convergence
result for strongly-convex functions. 
Using the Symplectic Euler discretization
scheme, we obtain an iterative optimization 
algorithm. 
We  have considered a discrete mean dissipation
restart scheme, but we have also
 introduced a new restart procedure based on
ensuring at each iteration a decrease of
the objective function greater than the one
achieved by a step of the classical gradient method.
For the discrete conservative algorithm,
this last restart criterion is capable of
guaranteeing a qualitative
 convergence result.
We apply the same restart scheme to the
Nesterov Accelerated Gradient (NAG-C), and we
use this restarted NAG-C as benchmark in the
numerical experiments.
In the smooth convex problems considered,
our method shows a faster convergence rate than
the restarted NAG-C.
We propose an extension of our discrete conservative
algorithm to composite optimization: in
the numerical tests involving non-strongly convex
functions with $\ell^1$-regularization,
 it has better
performances than the well known efficient
Fast Iterative Shrinkage-Thresholding Algorithm,
 accelerated with an adaptive restart scheme.

\end{abstract}

\maketitle

\begin{section}{Introduction}\label{sec:Int}
Convex optimization is of primary importance
in many fields of Applied Mathematics.
In this paper we are interested in unconstrained
minimization problems of the form
\begin{equation*}
\min_{x \in \R^n}f(x),
\end{equation*}
where $f:\R^n \to \R$ is a smooth convex
function. We will further assume that
$\nabla f$ is Lipschitz-continuous and that
$f$ is strongly convex.
The simplest algorithm for the numerical 
resolution of this minimization problem is the
classical gradient descent.
In the second half
of the last century
other important first-order algorithms were
introduced in order to speed up the
convergence of the gradient descent:
Polyak proposed his
\textit{heavy ball method} 
(see \cite{P64}, \cite{P87}), 
and Nesterov introduced a new class of 
\textit{accelerated gradient descent methods}
(see \cite{N83}, \cite{N18}). 
For a complete introduction to the subject,
we refer the reader to \cite{B14} and \cite{BV04}.

 The approach
of blending the study of optimization
methods with Dynamical Systems considerations
has been fruitfully followed in several
recent works, where Dynamical Systems tools
were employed to study existing optimization
methods and to introduce new ones:
in \cite{SBC} the authors derived an ODE for
modeling the Nesterov Accelerated Gradient
algorithm; in \cite{SJ18} and \cite{SJ19} the authors
studied accelerated methods (Nesterov and Polyak)
through {\it high-resolution} ODEs. Other
contributions in this direction come from
\cite{A16} and \cite{A18}.
Almost all the
ODEs obtained in the aforementioned papers can 
be reduced to the form
\begin{equation} \label{eq:cons_pert}
\ddot{x} + \nabla f(x) = -B(x,t) \dot{x},
\end{equation}where $B(x,t)$ is a symmetric
positive definite matrix, 
possibly depending on $t$. 
The term  $-B(x,t)\dot x$ in \eqref{eq:cons_pert}
represents the contribution of a
generalized viscous friction.
If, for example, the
matrix $B$ does not depend 
on the time $t$, 
then the convergence of any solution
of \eqref{eq:cons_pert} to the minimizer of $f$
is guaranteed by the dissipation of the total
mechanical energy
 $H = \frac12 |\dot{x}|^2 + f(x)$,
which plays the role of Lyapunov function.
Indeed, by differentiation of the energy $H$ 
along any solution of \eqref{eq:cons_pert}, we 
obtain
$
\frac{d}{dt}H(t) = -\dot x^T B(x) \dot x<0,
$
as long as $\dot{x} \neq 0$.
The choice of the matrix $B(x,t)$
is of primary importance
as shown in \cite{A18} and \cite{SBC}.

{
We recall that the Dynamical System approach 
{was} first undertaken in \cite{SBC}, where
{the authors proved that the Nesterov method
for non-strongly convex functions (NAG-C)
can be modeled by considering
the solutions of \eqref{eq:cons_pert},
with the dissipative term of the form 
$-B(x,t)\dot x =  -\frac3t \dot x$
and with initial velocity equal to zero.
Moreover, they proved that the objective
function achieves a decay $O(t^{-2})$ along
these curves. In order to avoid oscillations
of the solutions, which otherwise would slow down
the convergence, the authors introduced an adaptive
restart strategy that consists in resetting the
velocity equal to zero in correspondence of
local maxima of the kinetic energy $E_K:=\frac12 |
\dot x|^2$. In  \cite{SBC} it was shown that this 
adaptively restarted method  has a linear convergence 
rate when the objective function is strongly 
convex.}}
{In this paper we will  focus
on the conservative mechanical
ODE
\begin{equation} \label{eq:cons_ode_intr}
\ddot x + \nabla f (x) =0,
\end{equation}
and we will design a
piecewise conservative method based on
an adaptive restart strategy:
the resulting continuous-time algorithm achieves 
a linear convergence rate when the objective
function is strongly convex. 
As well as the aforementioned restarted method 
proposed in \cite{SBC}, our method
does  not make use of  the constant of strong 
convexity of the objective function. 
This is a relevant point, since in practice
estimating this quantity may be a very challenging
task. 

The presence of
viscosity friction in the 
dynamics studied in \cite{SBC}
yields to dissipate the kinetic energy
from the very beginning of the motion.
In alternative, if we consider 
dynamics \eqref{eq:cons_ode_intr} with initial 
velocity equal to zero
and starting point
$x(0)=x_0$ far from the minimizer $x^*$, it 
could be a good idea to let the 
system evolve without damping (i.e., conservatively)
 for an amount of time $\Delta T$,
so that the  solution
may be free to get closer 
to the minimizer, without being 
decelerated by the viscosity friction.
This is the idea that underlies 
the restarted  method that
we consider in this paper,
based  on the maximization of the
{\it mean dissipation} of the
kinetic energy as a stopping  criterion.

The introduction of this new restart strategy
is motivated by the fact that the maximization of the 
kinetic energy (employed in \cite{SBC})
is not suitable for the conservative
dynamics \eqref{eq:cons_ode_intr}.
Indeed, in \cite{SBC} the proof of the uniform upper
bound for the restart time (which is the cornerstone
of the linear convergence result) heavily relies
on the presence of the viscosity friction.
On the other hand, in the conservative dynamics,
the efficacy of the maximization of the
kinetic energy as restart criterion
 depends on the dimension
of the ambient space. 
In the one-dimensional case the solution
converges to a local minimizer
in a single restart iteration, as shown in
Section~\ref{sec:1d}.
Unfortunately, in the multi-dimensional case
 and for a general $f$, 
it is not possible to prove that a local
maximum of the kinetic energy 
is reached in a finite amount of time.
For this reason in \cite{TPL}, where the
conservative dynamics with maximization of the
kinetic energy was investigated, the authors
proved a linear convergence result for strongly
convex objective functions by assuming {\it a priori} 
the existence of a uniform upper
bound for the restart time.

The original restart strategy
(based on the {\it mean dissipation} of the kinetic
energy) that we develop
in Section~\ref{sec:alt} allows us to manage the
absence of viscosity friction in the system. 
Indeed, we can prove that the restart time is
always finite in the case that the objective function
is coercive.
Moreover, when dealing with strongly convex functions,
we prove that the restart time is uniformly
bounded, and hence we can strengthen the linear 
convergence result of \cite{TPL}. We also show that
the trajectory obtained has finite length.

In Section \ref{sec:discr} we derive a
discrete-time optimization algorithm by
applying the Symplectic Euler scheme
to \eqref{eq:cons_ode_intr}, yielding an update
rule of the form
$x_{k+1}= x_k -\alpha\nabla f(x_k) + (x_k - x_{k-1})$
where  $\alpha>0$ is the step-size. 

We introduce two restart criteria
for the discrete algorithm based
on the maximization of the mean dissipation of the
discrete kinetic energy,
referred as  RCM-mmd (Algorithm~1
and Algorithm~2). 
We also consider the case of restarting the discrete
conservative algorithm when the discrete kinetic
energy is maximized: this yields to the algorithm
investigated in \cite{SI}, where the authors
proved some partial results for quadratic
objective functions.

Finally, we design a restart criterion 
by imposing that, at each iteration, the decrease
of the objective function is greater or equal 
than the per-iteration-decrease achieved by the 
classical gradient descent method
with the same step-size.
We end up obtaining a discrete method referred as
RCM-grad (Algorithm~3), similar to
those described in \cite{TPL}.
Moreover, we observe that this
reasoning holds also for the Nesterov
Accelerated Gradient with the {\it gradient restart}
scheme (NAG-C-restart) proposed in \cite{OC}
and recently employed in \cite{K18}.
In other words, both RCM-grad and NAG-C-restart
achieve {\it at each iteration} an effective
acceleration of the gradient method.
This fact allows 
us, as a by-product, to prove a
{\it qualitative}
global convergence result for  RCM-grad:
{to the best of our knowledge, this
is the first convergence result for a method
based on the discretization of conservative
dynamics}.
This method is suitable both
for strongly and non-strongly convex optimization,
since it does not require an {\it a priori}
estimate of the strong convexity constant
of the objective function.
{We recall that other important 
contributions in this direction
for Nesterov-like restarted methods come from
\cite{N13}, \cite{FQ19}, \cite{FQ20}.}

In Section \ref{sec_num_test} 
{we have planned a quite extensive 
comparison to experimentally evaluate the performance 
of the convergence rate between the different 
discrete-time restart methods and
the  different versions of
the Nesterov Accelerated Gradient.}
In particular, we use as benchmark 
NAG-C-restart, since it was shown to
achieve high performances in both strongly and
non-strongly convex optimization (see \cite{OC}),
and it does not require the knowledge of the
constant of strong convexity.
We also give some insights on possible extensions
of our method for composite optimization problems.
We carry out numerical experiments
in presence of $\ell^1$-regularization and
we compare our method with the restarted FISTA
proposed in \cite{OC}.}
  
\end{section}

\begin{section}{One-dimensional case and
quadratic functions}\label{sec:1d}
{In this Section we introduce the
piecewise conservative method with restart 
strategy based on the maximization of the
kinetic energy $E_K:=\frac12 |\dot x|^2$.
More precisely, given a $C^{1,1}_L$-function
$f:\R^n \to \R$ 
(i.e., a function of class $C^1$ such that its 
gradient $\nabla f$ is Lipschitz-continuous
with constant $L>0$) to be minimized,
and given a starting point $x_0\in \R^n$, we consider
the solution of the ODE
\begin{equation} \label{eq:cons_ode_sec1d}
\ddot x + \nabla f = 0,
\end{equation}
with $x(0)=x_0$ and $\dot x(0) = 0$.
We recall that \eqref{eq:cons_ode_sec1d} preserves
the total mechanical energy $H(x,\dot x) :=
\frac12 |\dot x|^2 + f(x)$, and we observe that
at the initial instant the total mechanical
energy coincides with the potential energy.
Hence, during the motion,
part of the initial potential
energy is transformed into kinetic energy, and, if we
aim at minimizing $f$, a natural strategy could be
to wait until the kinetic energy attains a local
maximum. At this point, we reset the velocity 
equal to zero, and we repeat the whole procedure.
This continuous-time method was 
investigated in \cite{TPL}, where the authors proved
that, when $f$ is coercive and any critical point is
a minimizer, the set of the minimizers of $f$ is
globally asymptotically stable for the trajectories
of the restarted system.
However, the main difficulty in establishing
convergence rates lies in the estimate of the
restart time. As we show in this section, in the
one-dimensional case the system converges to a
local minimizer in a single iteration, i.e.,
in a finite amount of time.  
On the other hand, in the multi-dimensional case
the weakness of this criterion is inherent in
our inability to prove that the restart time is finite
for a generic strongly convex function.
In Section~\ref{sec:alt}
we modify the restart strategy in order to 
overcome this issue.}
{We use $\mathscr{S}^{1,1}_{\mu,L}$
to denote the 
class of functions in $C^{1,1}_L$ that are 
$\mu$-strongly convex, i.e., 
there exists $\mu >0$ such that
$x\mapsto f(x)-\frac{1}{2}\mu |x|^2$
is convex.}

We start by investigating the one-dimensional 
case, {when} $f: \R \to \R$ is a 
 {$C^{1,1}_L$} function. 
We consider the following Cauchy
problem:
\begin{equation} \label{eq:Cau_prob_1d}
\begin{cases}
\ddot{x} + f'(x) = 0, \\
x(0) = x_0, \\
\dot{x}(0) = 0,
\end{cases}
\end{equation}
and we reset the velocity equal to zero whenever the
kinetic energy $E_K = \frac12 |\dot x |^2$
achieves a local maximum.
We prove that this continuous-time method 
arrives to a local minimizer of $f$ at the
first restart.
  
\begin{proposition} \label{prop:1d_finite}
Let $f:\R \to \R$ be a 
{$C^{1,1}_L$} function 
and let us assume that $f$ is coercive.
For every $x_0 \in \R$ such that $f'(x_0) \neq 0$, 
let $x:[0,+\infty)\to \R$ be the solution of 
Cauchy problem \eqref{eq:Cau_prob_1d}.
Then, there exists $\bar{t} \in (0,+\infty)$ such 
that the kinetic energy function
$E_K: t \mapsto \frac12 | \dot{x}(t) |^2$ has a 
local maximum at $\bar{t}$.
Moreover, for every $\bar{t} \in (0,+\infty)$ 
such that $E_K$ has a local maximum at $\bar{t}$, 
the point
$x(\bar{t})$ is a local minimizer of $f$. 
\end{proposition}
The proof of Proposition~\ref{prop:1d_finite}
is postponed in Appendix~\ref{App_dim_1}.
Under the same assumptions and notations of Proposition~\ref{prop:1d_finite},
we can compute an explicit expression for the instant $\bar t$ when the solution of 
\eqref{eq:Cau_prob_1d} visits for the first time
the local minimizer  $x^* = x(\bar t)$.
We may assume 
that $x_0 <x^*$. For every $y \in [x_0, x^*]$ and for $t \in [0,\bar t ]$,
from the conservation of the total mechanical energy it follows that the solution
of \eqref{eq:Cau_prob_1d} visits the point $y$ with velocity 
$v_y = \sqrt{2(f(x_0)-f(y))}$. Thus, we obtain that
\begin{equation} \label{eq:time_1d}
\bar t = \int_{x_0}^{x^*}\frac{1}{\sqrt{2(f(x_0)-f(y))}}dy.
\end{equation} 
We observe that the hypothesis $f'(x_0) \neq 0$ guarantees that the singularity
at $x_0$ in \eqref{eq:time_1d} is integrable,
and thus that $\bar t$ is finite.

When the objective function $f: \R \to \R$ is 
{in $\mathscr{S}^{1,1}_{\mu,L}$}
we can give an upper bound to $\bar t$ that does not 
depend on the initial position $x_0$.
We prove this in the 
following Proposition.

\begin{proposition}\label{prop:est_time_1d}
Let $f:\R \to \R$ be a 
 function {in $\mathscr{S}^{1,1}_{\mu,L}$}.
Let $x^*$ be the unique minimizer of $f$ and let us choose $x_0 \in \R$
such that $x_0 \neq x^*$.
Let $t \mapsto x(t)$ be the solution of the Cauchy problem \eqref{eq:Cau_prob_1d}
and let $\bar t$ be the instant when the solution visits for the first time
the point $x^*$.
Then the following inequality holds:
\begin{equation} \label{eq:est_time}
\bar t \leq \frac{\pi}{2\sqrt{\mu}}.
\end{equation}  
\end{proposition}

The proof of Proposition~\ref{prop:est_time_1d}
is postponed in Appendix~\ref{App_2dim}.
The statement of Proposition~\ref{prop:est_time_1d} is sharp: inequality
\eqref{eq:est_time} is achieved for quadratic functions.
On the other hand, if the function $f$ is not strongly convex,
the visiting time $\bar t$ depends, in general, on the initial position.
As we are going to show in the following example, it may happen that
the closer the starting point is to the minimizer, the longer it takes
to arrive at.

\begin{example}
Let $f: \R \to \R$ be defined as $f(x)=\frac14 x^4$. Clearly, $f$ is strictly
convex (but not strongly) and $x^*=0$ is the unique minimizer. 
Let us choose $x_0>0$.
Then we have that
\begin{align*}
\bar t  = & \int_0^{x_0}\frac{\sqrt{2}}{\sqrt{x_0^4 - y^4}}dy 
=\int_0^{x_0}\frac{\sqrt{2}}{\sqrt{x_0^2 + y^2} \sqrt{x_0^2 - y^2}}dy \\
&\geq  \int_0^{x_0}\frac{1}{x_0 \sqrt{x_0^2 - y^2}}dy
= \frac{\pi}{2x_0}. 
\end{align*}
This shows that, in general, we can not give \textit{a priori}
an upper bound for $\bar t$. However, this does not mean that methods designed
with this approach are not suitable for the optimization of non-strongly
convex functions. Indeed, the visiting time $\bar t$ is finite, and this
guarantees that the continuous-time method converges in a finite amount of time.
This is not true, for example, in the case of the classical gradient flow.
\end{example}

The multidimensional case is much more complicated.
We now focus on quadratic objective functions
and, as we will see,
also in this basic case our global knowledge
is quite unsatisfactory.
On the other hand, the study of quadratic functions 
leads to useful considerations that we try
to apply to more general cases.
Let us consider the Cauchy problem
\begin{equation} \label{eq:Cau_prob_nd}
\begin{cases}
\ddot{x} + \nabla f(x) = 0, \\
x(0) = x_0, \\
\dot{x}(0) = 0.
\end{cases}
\end{equation}
The main difference with respect to the one-dimensional case lies in the fact
that, in general, the solution $t\mapsto x(t)$ of \eqref{eq:Cau_prob_nd} never
visits a local minimizer of $f$. The example below shows this phenomenon.

\begin{example}
Let us consider $f:\R^2 \to \R$ defined as 
$f(x_1, x_2) = \frac{a^2}{2}x_1^2 + \frac{b^2}{2} x_2^2$, where $a,\,b > 0$.
Let us set $ x(0) = (x_{0,1}, \, x_{0,2})^T\in \R^2 $. Then the solution of Cauchy
problem \eqref{eq:Cau_prob_nd} is 
\[
t \mapsto x(t) = (x_{0,1} \cos(at),\, x_{0,2} \cos (bt))^T.
\]
If $x_{0,1}, \, x_{0,2} \neq 0$ and if the ratio $a / b $ is not a rational number,
then $x(t) \neq (0,0)^T$ for every $t\in [0, +\infty )$.
This also shows that, when the dimension is larger than one,
Proposition~\ref{prop:1d_finite} fails. Indeed, it is easy to check that
the kinetic energy function $t \mapsto \frac12 |\dot{x}(t)|^2$ has many 
local maxima, but the solution never visits any local minimizer of $f$.  
\end{example}

When $f: \R^n \to \R$ is a strongly convex quadratic function, we can estimate
the decrease of the objective function after each arrest.

\begin{lemma} \label{lem:est_decr}
Let $f: \R^n \to \R$ be a quadratic function of the form
\[
f(x) = \frac12 x^T A x,
\]
where $A$ is a symmetric and positive definite matrix. Let $x_0 \in \R^n$
be the starting point of Cauchy problem \eqref{eq:Cau_prob_nd}. Let
$0<{ \lambda}_1\leq \ldots \leq { \lambda}_n$ be the eigenvalues of $A$.
Then, the following inequality is satisfied:
\begin{equation} \label{eq:est_decreas}
f\left( x \left(  \frac{\pi}{2\sqrt{{ \lambda}_n}}  \right)  \right) \leq
\cos^2 \left(  \frac{\pi}{2} \sqrt{\frac{{ \lambda}_1}{{ \lambda}_n}}  \right)
f(x_0).  
\end{equation}
\end{lemma}

\begin{remark}
Let us assume that the kinetic energy function has at least one local maximizer
and let  $t_1 \in (0, +\infty)$ be the smallest.
Then Lemma~\ref{lem:est_decr} implies that 
\[
f(x(t_1)) \leq \cos^2 \left(  \frac{\pi}{2} \sqrt{\frac{{ \lambda}_1}{{ \lambda}_n}}  \right)f(x_0).
\]
Indeed, we have that 
$t_1 \geq \frac{\pi}{2\sqrt{{ \lambda}_n}}$, 
since the time derivative of the
kinetic energy function is non-negative at $t=\frac{\pi}{2\sqrt{{ \lambda}_n}}$.  
Hence, if we iterate the evolution-restart procedure $k$ times
(assuming that the kinetic energy function always attains a local maximum)
 and if we
call $x^{(k)}$ the restart point after the $k$-th iteration, we have that
\[
f(x^{(k)}) \leq 
\left[ \cos^2 \left(  \frac{\pi}{2} \sqrt{\frac{{ \lambda}_1}{{ \lambda}_n}}  \right) \right]^k f(x_0).
\]
So, in terms of evolution-restart iterations, we have that the value of the
objective function decreases at exponential rate.
However, since we do not have an upper bound on the 
restart time, we
do not know the rate of decrease in terms of the 
evolution time. As we explain in the next Section,
 we can overcome this problem
by designing alternative restart criteria.
For example, in the particular case of
quadratic functions,
we can keep the free-evolution amount of time
constant and equal to 
$\Delta T = \frac{\pi}{2 \sqrt{{ \lambda}_n}} $.
Let $t\mapsto \tilde x (t)$ be the curve obtained
with this procedure,
then, owing to the proof of 
Lemma~\ref{lem:est_decr}, we
have that
\[
f(\tilde x(t)) \leq \left[ \cos^2 \left(  \frac{\pi}{2} \sqrt{\frac{{ \lambda}_1}{{ \lambda}_n}}  \right) \right]^{\left\lfloor\frac{t}{\Delta T} \right\rfloor} f(x_0),
\] 
where  $\left\lfloor \cdot \right\rfloor$ 
denotes the integer part. 
\end{remark}

\end{section}

\begin{section}{An alternative restart criterion} \label{sec:alt}
The restart criterion that we have considered so far consists in waiting
until the kinetic energy reaches a local maximum.
This idea has already been introduced in
\cite{SBC} in order to improve the convergence
rate of the solutions of the ODE modeling
the Nesterov method.
{Indeed, the solutions of the ODE 
considered in \cite{SBC} exhibit 
undesirable oscillations, that can be avoided
by means of this adaptive restart strategy.
As a matter of fact, in \cite{SBC} it is proved
that the restarted Nesteov ODE achieves a linear
convergence rate when the objective function is
strongly convex.
 The proof of this result consists
basically of two main steps (i.e., the estimate
of the decrease of the objective function at each
evolution-restart iteration, and the uniform
upper bound for the restart time), and in both of them
the presence of the viscosity friction plays
a crucial role. On the other hand, in the conservative
ODE that we study in the present paper there is no
viscosity term, hence we cannot adapt to our case
the arguments employed in \cite{SBC}.  In order to
manage the absence of friction, in this section we
introduce an original restart strategy, alternative
to the one considered so far.
We prove that the resulting continuous-time method
achieves linear convergence rate when the objective
function is strongly convex, and we show that
the curve obtained has finite length. 
It is important to recall that our method, as well
as the restarted method studied in \cite{SBC},
does not make use of the constant of strong
convexity of the objective function.}

In this section we propose a restart criterion based 
on the maximization of the mean dissipation.
If we arrest the conservative evolution at the instant 
$t>0$, then the
value of the kinetic energy $E_K(t)$ at the instant 
$t$ equals the decrease of the objective function. 
The idea behind this alternative restart criterion 
is that we arrest the conservative evolution
of the system when the mean dissipation
$t\mapsto E_K(t)/t$ reaches a local maximum.
{We use the notation $C^{2,1}_L$ to
denote the functions in $C^2$ whose 
gradient is Lipschitz with constant $L>0$.
The symbol $\mathscr{S}^{2,1}_{\mu,L}$ is used to 
indicate the functions in $C^{2,1}_L$ that are
strongly convex with constant $\mu>0$.}

Let $f:\R^n \to \R$ be a 
{$C^{1,1}_L$} function and 
let us consider the Cauchy
problem
\begin{equation} \label{eq:Cau_prob_nd_alt}
\begin{cases}
\ddot{x} + \nabla f(x) = 0, \\
x(0) = x_0, \\
\dot{x}(0) = 0.
\end{cases}
\end{equation}
Let us define the function $r:[0,+\infty)\to [0,+\infty)$ as
\begin{equation} \label{eq:Mean_Diss}
r(t) = \begin{cases} 0 & t=0, \\ \frac{E_K(t)}{t} & t>0,   \end{cases}
\end{equation}
where $E_K$ is the kinetic energy function relative to the solution of
Cauchy problem \eqref{eq:Cau_prob_nd_alt}.
We observe that $r$ is differentiable at $t=0$, since we have that
\begin{equation}\label{eq:expans_E_K}
E_K(t) = \frac12|\nabla f (x_0)|^2t^2 + o(t^2), 
\end{equation}
as $t\to0$. 
If we take the derivative  of $r$ with respect to the time, we obtain that
\begin{equation} \label{eq:der_mean_diss}
\frac{d}{dt}r = \frac{t\dot{E}_K(t) - E_K(t)}{t^2}
\end{equation}
for every $t>0$. With a simple computation, we can check that the derivative of $r$
can be continuously extended at $t=0$. We have that
\begin{equation*} 
\frac{d}{dt}r = 
\begin{cases} 
\frac12|\nabla f (x_0)|^2 & t=0, \\
\frac{t\dot{E}_K(t) - E_K(t)}{t^2} & t>0.
\end{cases}
\end{equation*}
We observe that the derivative of $r$ at $t=0$ is positive, hence
it remains non-negative in an interval $[0,\varepsilon)$.
The Maximum Mean Dissipation criterion consists of restarting the evolution when
the function $t\mapsto r(t)$ reaches a local maximum.
The restart time is 
\begin{equation} \label{eq:MMD_time}
 t_a = \inf \{    t:\,  t\dot{E}_K(t) - E_K(t)<0  \}.
\end{equation}
We observe that, if $\bar t \in (0, +\infty)$ is a local maximizer of the kinetic 
energy, then we have that
\[
 \bar t \dot{E}_K(\bar t) - E_K(\bar t) = -E_K(\bar t) <0. 
\]
This means that a local maximizer of the kinetic energy can not be a maximizer
of the mean dissipation $r$. This fact is described in  Figure~\ref{fig:MMD}.
\begin{center}
\begin{figure}
\includegraphics[scale=0.22]{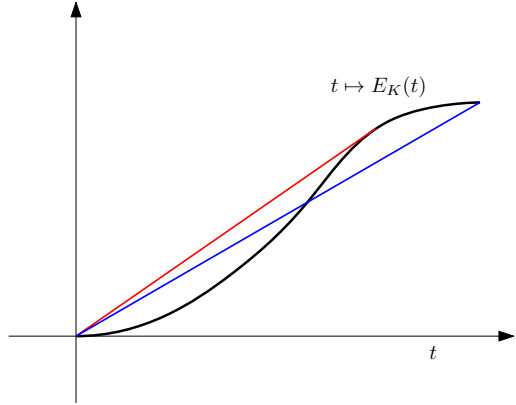}
\caption{Mean dissipation. The black graph represents a typical profile of the
kinetic energy function, in the case it attains a local maximum.
The slope of the segments represents the mean dissipation that we obtain
when we stop the evolution in a given instant. The picture shows that
stopping the evolution in correspondence of a local maximum of the 
kinetic energy function does not guarantee the highest mean dissipation.}
\label{fig:MMD}
\end{figure}
\end{center}
We can prove that the restart time $t_a$
is finite. We remark that the following 
result holds even if the function $f$ is not
convex.
\begin{lemma} \label{lem:MAD_time}
Let $f: \R^n \to \R$ be a
{$C^{1,1}_L$} coercive function 
and let us take $x_0 \in \R^n$.
Let $t\mapsto E_K(t)$ be the kinetic energy function of the solution of
Cauchy problem \eqref{eq:Cau_prob_nd_alt}. Then there exists 
$\hat{t} \in (0,+\infty)$ such that
\[
\hat{t}\dot{E}_K(\hat{t}) - E_K(\hat{t}) <0.
\]
\end{lemma}

\begin{proof}
We argue by contradiction. Let us assume that
\begin{equation} \label{eq:abs_ineq}
t\dot{E}_K(t) - E_K(t) \geq0,
\end{equation}
for every $t\geq 0$.
Using \eqref{eq:der_mean_diss}, we deduce that
the mean dissipation $t \mapsto r(t)$
is non-decreasing for every $t\geq 0$.
Let $t_1>0$ be any instant such that the kinetic 
energy is positive, i.e., $E_K(t_1)>0$.
Then we have that
\begin{equation} \label{eq:est_ener_time}
E_K(t) \geq \frac{E_K(t_1)}{t_1}t,
\end{equation}
for every $t>t_1$. This is impossible since the kinetic energy is
always bounded from above if the
function $f$ is coercive.  
\end{proof}

Using the idea of the proof of Lemma~\ref{lem:MAD_time},
we can estimate from above the restart time of the Maximum Mean Dissipation.
Indeed, the derivative of the mean dissipation is positive at $t=0$, and then
it remains non-negative in the interval $[0, t_a]$. Then, using the same notations
as in the proof above,
for every $t \in [t_1, t_a]$
the kinetic energy function $E_K$ satisfies inequality \eqref{eq:est_ener_time}.
On the other hand, from the conservation of the energy 
{it} follows that
\[
f(x_0) - f^* \geq E_K(t),
\]
where $f^*$ is the minimum value of
 the objective function $f$.
This implies that
\begin{equation*}
f(x_0) - f^* \geq \frac{E_K(t_1)}{t_1}t_a,
\end{equation*}
that can be rewritten as
\begin{equation} \label{eq:time_est_provv_1}
t_a \leq \frac{t_1}{E_K(t_1)}(f(x_0) - f^*).
\end{equation}
{
\begin{remark}
Inequality \eqref{eq:time_est_provv_1} implies that
the restart condition of the Maximum Mean Dissipation
is met after a finite amount of time, as soon as
$f$ is a coercive function in $C^{1,1}_L$.
On the other hand, to the best of our knowledge,
it is not possible to exclude that the kinetic
energy  $E_K$ could grow monotonically, without 
assuming maximum. This means that in general,
when dealing with conservative system
\eqref{eq:Cau_prob_nd_alt} together with the
 restart strategy based on the
maximization of the kinetic energy, it
is not possible to prove an upper bound for the
restart time. 
For this reason, in \cite{TPL}
the linear convergence result for strongly convex 
objective functions is proved under the assumption
of the {\it a priori} existence of a uniform upper 
bound for the restart time.
On the other hand, using the Maximum Mean Dissipation,
when the objective function is in 
$\mathscr{S}^{1,1}_{\mu,L}$ we can bound the
restart time with a quantity that depends only on
the constants $\mu$ and  $L$.      
\end{remark}}
In the case of a strongly convex function, we can give 
{ uniform}
estimates for the restart time $t_a$.

\begin{proposition} \label{prop:stop_time_low_up}
Let $f:\R^n \to \R$ be a  
function {  in $\mathscr{S}^{1,1}_{\mu,L}$}.
For every $x_0\in \R^n$, let us consider Cauchy 
Problem 
\eqref{eq:Cau_prob_nd_alt} with starting point $x_0$ 
and let
$t\mapsto E_K(t)$ be the kinetic energy function of 
the solution.
Let $t_a$ be the stopping time defined in 
\eqref{eq:MMD_time}.
Then the following estimate{ s} hold:
\begin{equation}\label{eq:stop_time_low}
t_a > \frac{\sqrt{\mu}}{8{ L}},
\end{equation}
{ and
\begin{equation} \label{eq:stop_time_up}
t_a \leq \mathcal{T}_{R} := 32 \frac{L}{\mu\sqrt{\mu}}.
\end{equation}}
\end{proposition}

We postpone the proof of 
Proposition~\ref{prop:stop_time_low_up}
since we need some technical lemmas.
In the following lemma we recall the
Polyak-Lojasiewicz inequality for strongly-convex
functions (see \cite{P63}{ , and
\cite[Theorem~2.1.10]{N18} for its proof}).

\begin{lemma} \label{lem:est_f_SC}
Let $f:\R^n \to \R$ be a
{ $C^1$} function { 
and let us assume that $f$ is $\mu$-strongly convex,
with $\mu>0$}. 
Let $x^*$ be the unique 
minimizer of $f$.
Then, for every $x \in \R^n$,  the following 
Polyak-Lojasiewicz inequality holds:
\begin{equation} \label{eq:grad_vs_f_SC}
f(x)-f(x^*)\leq \frac{1}{2\mu}|\nabla f (x)|^2.
\end{equation}
\end{lemma}

In the following lemma, we give an estimate of the growth of the
kinetic energy function $t\mapsto E_K(t)$ when $t$ is small.

\begin{lemma} \label{lem:kin_en_small_time}
Let $f:\R^n \to \R$ be a 
function in { $\mathscr{S}^{1,1}_{\mu,L}$}.
For every $x_0\in \R^n$, let us consider
 Cauchy Problem 
\eqref{eq:Cau_prob_nd_alt} with starting point
 $x_0$ and let $t\mapsto E_K(t)$ be the 
 kinetic energy function of the solution.
Then, for every 
$0\leq t\leq \sqrt{\mu} / (2{ L})$,
 the following inequality holds:
\begin{equation}\label{eq:E_Kin_small_time}
\frac{1}{8}| \nabla f (x_0) |^2 t^2 \leq E_K(t)
 \leq \frac{25}{32}| \nabla f (x_0) |^2 t^2.
\end{equation}
\end{lemma} 
\begin{proof}
We recall that ${ L}$ is a Lipschitz 
constant for $x\mapsto \nabla f(x)$.
Using the conservation of the mechanical energy,
we deduce that
\begin{align*}
|\nabla f (x(t)) - \nabla f(x_0)| &\leq 
{ L} |x(t) - x_0|
\leq { L} \int_0^t|\dot{x}(u)| \,du \\
&\leq { L}
 \int_0^t \sqrt{2}\sqrt{f(x_0) -f^*} \,du 
=  { L}\sqrt{2}\sqrt{f(x_0) -f^*} \,t.
\end{align*}
Owing to \eqref{eq:grad_vs_f_SC}, we
obtain that
\begin{equation}\label{eq:grad_small_time}
|\nabla f (x(t)) - \nabla f(x_0)| \leq 
\frac{{ L}}{\sqrt{\mu}} |\nabla f(x_0)|t.
\end{equation}
Using this fact, we deduce that
\begin{equation*}
|\dot{x}(t) + t\nabla f(x_0) |=
\left| \int_0^t (-\nabla f (x(s)) + \nabla f(x_0))
 \,ds \right|
\leq  
\int_0^t \frac{{ L}}{\sqrt{\mu}}
|\nabla f(x_0)|s \,ds.
\end{equation*}
Hence we have that
\begin{equation} \label{eq:ineq_techn_lem_1}
|\dot{x}(t) +t\nabla f(x_0) | \leq 
\frac{{ L}}{2\sqrt{\mu}} |\nabla f(x_0)|t^2.
\end{equation}
Using \eqref{eq:ineq_techn_lem_1} and the 
triangular inequality, we obtain that
\begin{equation*}
| \nabla f(x_0) |t - \frac{{ L}}{2\sqrt{\mu}}
| \nabla f(x_0) |
t^2 \leq
 |\dot{x}(t)| \leq | \nabla f(x_0) |t 
 + \frac{{ L}}{2\sqrt{\mu}}|\nabla f(x_0)|t^2.
\end{equation*}
Therefore, if $t\leq \sqrt{\mu} / { L}$,
we have that
\[
\frac{1}{2} | \nabla f(x_0) |t \leq |\dot{x}(t)|.
\]
On the other hand, if 
$t\leq\sqrt{\mu} / (2{ L})$,
we have that
\[
|\dot{x}(t)| \leq \frac{5}{4} | \nabla f(x_0) |t.
\]
This concludes the proof.  
\end{proof}

We now prove 
Proposition~\ref{prop:stop_time_low_up}.

\begin{proof}[Proof of Proposition~\ref{prop:stop_time_low_up}]
The proof of inequality 
\eqref{eq:stop_time_low}
 is based on the study of the sign of the quantity 
$t\mapsto t\dot{E}_K(t)-E_K(t)$.
First of all, we observe that 
\begin{equation}
\dot{x}(t) = -t\nabla f(x_0) -\int_0^t\left(\nabla f(x(s))- \nabla f(x_0)  \right)\, ds.
\end{equation}
Therefore, we deduce that
\begin{align*}
\dot{E}_K(t) &= \ddot{x}(t)\cdot\dot{x}(t) = -\nabla f(x(t))\cdot\dot{x}(t) \\
&= \nabla f(x(t)) 
\cdot
\left( t\nabla f(x_0) 
+\int_0^t\left(\nabla f(x(s))- \nabla f(x_0)  \right)\, ds \right)\\
&=(\nabla f(x(t))-\nabla f(x_0))\cdot
\left( t\nabla f(x_0) 
+\int_0^t\left(\nabla f(x(s))- \nabla f(x_0)  \right)\, ds \right)\\
& \,\,\,\,\,\,\,\, +|\nabla f(x_0)|^2t + \nabla f(x_0) \cdot  
\int_0^t\left(\nabla f(x(s))- \nabla f(x_0)  \right)\, ds.
\end{align*}
Owing to \eqref{eq:grad_small_time}, we obtain that:
\begin{equation*}
\dot{E}_K(t) \geq  |\nabla f(x_0)|^2t
-\frac{3}{2}\frac{{ L}}{\sqrt{\mu}}
|\nabla f(x_0)|^2t^2 
-\frac12 
\frac{{ L}^2}{\mu}|\nabla f(x_0)|^2t^3.
\end{equation*}
Using inequality \eqref{eq:E_Kin_small_time}, we have that
\[
t\dot{E}_K(t)-E_K(t) \geq |\nabla f(x_0)|^2t^2
\left( 
\frac{7}{32} 
-\frac{3}{2}\frac{{ L}}{\sqrt{\mu}}t 
-\frac12 \frac{{ L}^2}{\mu}t^2
 \right),
\]
for $t\leq {\sqrt{\mu}}/({2{ L}})$.
With a simple computation, we obtain that $t\dot{E}_K(t)-E_K(t)>0$
when $t\leq \tilde{t}$, where
\begin{equation} \label{eq:t_tilda}
\tilde{t}:= \frac{\sqrt{\mu}}{8{ L}}.
\end{equation}
By the definition of the stopping time $t_a$, we deduce that 
$t_a> \tilde{t}$.
This proves { 
\eqref{eq:stop_time_low}.

We now prove \eqref{eq:stop_time_up}.
Using \eqref{eq:stop_time_low} 
and the definition of the
stopping time $t_a$, we have that
\[
\frac{E_K(t_a)}{t_a}\geq\frac{E_K(\tilde{t})}{\tilde{t}}, 
\]
where $\tilde{t}$ is defined in \eqref{eq:t_tilda}.
The last inequality can be rewritten as
\begin{equation} \label{eq:t_stop_est_0}
t_a \leq \frac{E_K(t_a)}{E_K(\tilde{t})} {\tilde{t}}.
\end{equation}
Using the conservation of the energy and inequality
\eqref{eq:E_Kin_small_time},
we obtain that
\begin{equation} \label{eq:est_E_Kin_ratio}
\frac{E_K(t_a)}{E_K(\tilde{t})} \leq 
\frac{8(f(x_0)-f^*)}{|\nabla f(x_0)|^2\tilde{t}^2}
\leq \frac{4}{\mu}\frac{1}{\tilde{t}^2},
\end{equation}
where in the second inequality we used Lemma \ref{lem:est_f_SC}.
Combining \eqref{eq:t_stop_est_0} and \eqref{eq:est_E_Kin_ratio},
and using the definition of $\tilde{t}$ given in \eqref{eq:t_tilda},
we obtain that $t_a \leq \mathcal{T}_{R}$, where
we set
\[
\mathcal{T}_{R} := 32 \frac{L}{\mu\sqrt{\mu}}. 
\]
This concludes the proof.}  
\end{proof}

{ The framework for analyzing our 
restarted scheme is similar to that employed
in \cite{SBC}, however
we point out that the proofs are different
since the ODE considered in \cite{SBC} and the 
conservative system \eqref{eq:Cau_prob_nd_alt}
are structurally different.}

For this stopping criterion we have proved that
 the restart time is uniformly bounded by 
 $\mathcal{T}_{R}$.
In the following result, we provide an estimate
about the value of the kinetic energy at the
restart instant. 

\begin{lemma} \label{lem:E_K_stop}
Let $f:\R^n \to \R$ be a 
{ $C^{2,1}_L$}  function.
For every $x_0\in \R^n$, let us consider Cauchy 
Problem 
\eqref{eq:Cau_prob_nd_alt} with starting point 
$x_0$ and let
$t\mapsto E_K(t)$ be the kinetic energy function 
of the solution.
Then the following inequality holds:
\begin{equation}
E_K(t_a) \geq \frac{1}{2{ L}}
| \nabla f (x(t_a))|^2,
\end{equation}
where $t_a$ is the stopping time defined in 
\eqref{eq:MMD_time}.
\end{lemma} 
\begin{proof}
Owing to the definition, we have that $t_a$ is a local 
maximizer for the function $t\mapsto r(t)$, where 
$r:[0,+\infty) \to [0,+\infty)$ is the Mean Dissipation 
function defined in \eqref{eq:Mean_Diss}.
Recalling that $t_a\dot{E}_K(t_a)=E_K(t_a)$,
 we have that
\[
\frac{d^2}{dt^2}r(t_a)= \frac{\ddot{E}_K(t_a)}{t_a} \leq 0.
\]
On the other hand, we have
\begin{equation}\label{eq:ineq_ddEk}
\ddot{E}_K(t_a) = |\nabla f (x(t_a))|^2 
- \dot{x}(t_a)^T \nabla ^2 f (x(t_a))\dot{x}(t_a) \leq 0. 
\end{equation}
By the hypothesis, the matrix 
${ L}\mbox{Id} - \nabla^2f(x)$
is positive definite
for every $x\in \mathbb{R}^n$.
Using this fact in \eqref{eq:ineq_ddEk},
we obtain that
\[
2{ L}E_K(t_a) - |\nabla f(x(t_a))|^2\geq 0,
\]
and this concludes the proof.  
\end{proof}

We conclude this section providing an estimate
about the decrease of the objective function
with the convergence result.
{  Moreover, we prove that the
piecewise conservative method
produces a curve that has finite length.}

\begin{theorem} \label{thm:conv_cont}
Let $f:\R^n \to \R$ be a 
function {  in $\mathscr{S}^{2,1}_{\mu,L}$}. 
Let $x^* \in \R^n$ be the unique minimizer of $f$,
and  let $x_0 \in \R^n$ be the starting point.
Let $t\mapsto \tilde{x}(t)$ be the curve obtained 
applying the following iterative procedure:
\begin{itemize}
\item set $t_0 = 0$ and consider the forward
solution of 
$\ddot{\tilde{x}} + \nabla f(\tilde{x}) = 0$, with
$\tilde{x}(t_0) = x_0$ and $\dot{\tilde{x}}(t_0) =0$;
\item for every $k\geq 1$, let $t_k$ be the instant
when the mean dissipation
\[
t \mapsto \frac{|\dot{\tilde{x}}(t)|^2}{2(t-t_{k-1})},
 \,\,\,\,\, t>t_{k-1}
\]
attains a local maximum for the first time,
and set $x_k=\tilde{x}(t_k)$.
Then consider the forward
solution of 
$\ddot{\tilde{x}} + \nabla f(\tilde{x}) = 0$, with
$\tilde{x}(t_k) = x_k$ and $\dot{\tilde{x}}(t_k) =0$.  
\end{itemize}
Then for every $t\geq 0$ 
 the following inequality is satisfied:
\begin{equation}\label{eq:conv_cont_thm}
f(\tilde{x}(t)) - f(x^*) \leq
\left(  1 + \frac{\mu}{{ L}}
 \right)^{-\left\lfloor \frac{t}{\mathcal{T}_R} 
 \right\rfloor} (f(x_0)-f(x^*)),
\end{equation}
where $\mathcal{T}_R$ is defined in 
\eqref{eq:stop_time_up}.
Moreover, we can prove the following
upper bound for the length of the curve 
$t\mapsto \tilde x(t)$:
\begin{equation} \label{eq:length_bound}
\int_0^\infty |\dot{\tilde x}(t)| \, dt
\leq 
4\sqrt{2}  \frac{L}{\mu}\mathcal{T}_R
\sqrt{f(x_0)-f(x^*)}.
\end{equation}
\end{theorem}
\begin{proof}
We begin by proving \eqref{eq:conv_cont_thm}.
Let $t_1 >0$ be the first stopping instant.
Owing to the conservation of the total mechanical energy,
we have that
\[
f(x_0) - f(x(t_1)) = E_K(t_1).
\]
On the other hand, combining Lemma~\ref{lem:E_K_stop} and
Lemma~\ref{lem:est_f_SC}, we obtain that
\[
E_K(t_1) \geq \frac{1}{2{ L}}
|\nabla f(x(t_1))|^2 \geq \frac{\mu}{{ L}}
\left( f(x(t_1))-f(x^*) \right).
\]
Therefore, we deduce that
\begin{equation} \label{eq:dec_per_iter}
f(x(t_1))-f(x^*) \leq \left(  1 + 
\frac{\mu}{{ L}} \right)^{-1}
(f(x_0)-f(x^*)).
\end{equation}
The last inequality gives an estimate of the decrease-per-iteration
of the objective function.
Owing to Proposition~\ref{prop:stop_time_low_up},
 we have that the stopping
time is always bounded by $\mathcal{T}_R$.
This means that in the interval $[0,t]$
the number $k$ of restart iterations is greater or equal than
$\left\lfloor \frac{t}{\mathcal{T}_R} \right\rfloor$, where $\left\lfloor 
\cdot \right\rfloor$ denotes the integer part.
Using inequality \eqref{eq:dec_per_iter}, we obtain that 
\begin{align*}
f(\tilde{x}(t)) -f(x^*) 
&\leq \left(  1 + \frac{\mu}{{ L}}
 \right)^{-k}(f(x_0)-f(x^*))\\
& \leq \left(  1 + \frac{\mu}{{ L}}
 \right)^{-\left\lfloor \frac{t}{\mathcal{T}_R} \right\rfloor}
(f(x_0)-f(x^*)).
\end{align*}
{  This proves \eqref{eq:conv_cont_thm}.

We now study the length of the curve $t\mapsto
\tilde x(t)$ at each evolution interval 
$[t_k,t_{k+1}]$ for $k\geq 0$.
Using the conservation of the total mechanical 
energy and Proposition~\ref{prop:stop_time_low_up},
 we have that
\begin{equation} \label{eq:lenght_piece}
\int_{t_k}^{t_{k+1}}|\dot{\tilde{x}}(t)| \,dt
\leq \mathcal{T}_R \sqrt{2(f(\tilde x(t_k)
-f(x^*)}.
\end{equation}
Moreover, owing to \eqref{eq:dec_per_iter}, we obtain
that
\begin{equation} \label{eq:dec_it_k}
f(\tilde x(t_k)
-f(x^*) \leq 
\left(  1 + \frac{\mu}{L} \right)^{-k}
(f(x_0)-f(x^*))
\end{equation}
for every $k\geq 0$.
Combining \eqref{eq:lenght_piece} and
\eqref{eq:dec_it_k}, we have that
\begin{equation*}
\int_0^\infty |\dot{\tilde{x}}(t)|\,dt =
\sum_{k=0}^\infty \int_{t_k}^{t_{k+1}} |\dot{\tilde{x}}(t)|\,dt
 \leq 
4\sqrt{2}  \frac{L}{\mu}\mathcal{T}_R
\sqrt{f(x_0)-f(x^*)}.
\end{equation*}} 
This concludes the proof.  
\end{proof}

\begin{remark} In the case of quadratic functions,
we can compare our convergence result with the one
proved in \cite{TPL}.
Let $f:\R^n \to \R$ be of the form
\[
f(x) = \frac12 x^TAx,
\]
where $A\in \R^{n\times n}$ is symmetric and 
positive definite, and let
$0<\lambda_1 \leq \ldots \leq \lambda_n $ be the
eigenvalues of $A$.
Let $x_0 \in \R^n$ be the starting point and
let $t \mapsto \bar x(t)$ the curve obtained
following the construction proposed in \cite{TPL}.
Owing to Theorem~2 and Lemma~4 in \cite{TPL},
the following estimate holds:
\[
f(\bar x (t)) \leq 
\left(  1 + \frac{{ \lambda_1}}{
{ \lambda_n}} \right)^{-\left\lfloor \frac{t}{\mathcal{T}_R'} \right\rfloor}
(f(x_0)-f(x^*))
\]
where
\[
\mathcal{T}_R' = \frac{2n\pi}{\sqrt{\lambda_1}}.
\]
On the other hand, if we consider the curve 
$t \mapsto \tilde x (t)$ obtained with our
restart procedure, inequality
\eqref{eq:conv_cont_thm} holds with
\[
\mathcal{T}_R =32 \frac{ \lambda_n}{\lambda_1 \sqrt{\lambda_1}}.
\]
Hence we observe that 
{ $\mathcal{T}_R'$} is affected by the
dimension of the problem, while 
{ $\mathcal{T}_R$} is sensitive
to the condition number of the matrix $A$. 
\end{remark}

{ 
\begin{remark}
It is important to observe that
the estimate expressed in \eqref{eq:length_bound}
is invariant if we multiply 
the objective function $f$ 
by a factor $\nu>0$. This is not the case for
the estimate in \eqref{eq:conv_cont_thm},
because the multiplication of $f$ by
a factor $\nu>0$ does affect the parametrization
of the curve produced by the method, while
the trajectory remains unchanged.
For these reasons, we believe that,
when dealing with continuous-time optimization
methods,  the study of the
length of trajectories may be a useful tool.
\end{remark}}

\end{section}

\begin{section}{Discrete version of the method} \label{sec:discr}
In this section we develop a discrete version of the 
continuous-time
algorithm that we have described so far.
The basic idea is to rewrite the second order ODE
\[
\ddot{x}+\nabla f (x) =0, 
{ \,\,\,x(0) =x_0,
\,\,\, \dot x(0)=0,}
\]
as a first order ODE, by doubling the variables:
\begin{equation}
\begin{cases} \label{eq:ODE_cont}
\dot{x}=v, &{   x(0)=x_0,} \\
\dot{v}=-\nabla f (x), &{   v(0)=0}.
\end{cases}
\end{equation}
The differential equation \eqref{eq:ODE_cont} 
is a time-independent Hamiltonian system
and for its discretization we use
the Symplectic Euler scheme, due to its well-known
suitability (see e.g. \cite{Hairer,SJ19}), 
yielding to the following recurrence sequence
{  for $k\geq0$ given $(x_0,v_0)$}:
\begin{equation}
\begin{cases} \label{alg:cons}
v_{k+1}= v_k  -h\,\nabla f (x_k), \\
x_{k+1}= x_k +h\,v_{k+1},
\end{cases}
\end{equation}
where $h>0$ is the discretization step, and where $x_0$ 
is the starting point and $v_0=0$.
We recall that,
in general, the Symplectic Euler scheme for
time-independent Hamiltonian systems leads to
implicit discrete systems. However, for the
particular Hamiltonian function 
$\mathcal{H}(x,v)= \frac12 |v|^2 + f(x)$,
the discrete system \eqref{alg:cons} is
explicit.
Combining the equations of \eqref{alg:cons},
we have that {  for $k\geq 0$}
\begin{equation} \label{eq:upd_rule}
x_{k+1}= x_k - h^2 \nabla f (x_k) + hv_k.
\end{equation}
This shows that the sequence defined in 
\eqref{alg:cons} consists of an iteration
 of the classical gradient descent method with 
step $h^2$, plus the momentum term $hv_k$.
\begin{remark}
We observe that we can
rewrite \eqref{eq:upd_rule} as follow{ s}
for  { $k\geq 0$}:
\begin{equation} \label{eq:upd_rule_x}
x_{k+1}
= x_k - h^2 \nabla f (x_k) + (x_k -x_{k-1}),
\end{equation}
with {  $x_{-1}=x_0 - hv_0$}.

\noindent
The last expression is very similar to the
update rule of the heavy-ball method:
\[
x_{k+1}
= x_k - \alpha  \nabla f (x_k) + \beta(x_k -x_{k-1}).
\]
It is important to recall that the 
local-convergence result
for the heavy-ball method proved in 
\cite{P87} requires that $0\leq \beta<1$.
This means that we can not 
apply the aforementioned theorem to
the sequence obtained using \eqref{eq:upd_rule}.
However, this is not an issue, since,
as well as in the continuous-time case, the
convergence of our method relies
on a proper restart scheme.  
\end{remark}

{ 
\begin{remark}
We point out that the discretization of 
conservative system \eqref{eq:ODE_cont}
should not be understood as a method for providing
an accurate approximation of the continuous-time
solution. Nevertheless, it is natural to ask what
is the structure of the update rule obtained with
an higher-order symplectic scheme.
For instance, if we apply the symplectic second-order
St\"ormer-Verlet scheme
(see \cite[Chapter VI]{Hairer})
to the conservative system \eqref{eq:ODE_cont},
a simple computation shows that the sequence $(x_k)_k$
produced  satisfies the so called 
{\it leapfrog scheme}, i.e., the same recurrence
scheme as in \eqref{eq:upd_rule_x}, with
initialization 
$x_{-1} = x_0 - hv_0 - \frac{h^2}{2}\nabla
f(x_0)$.
In other words, we obtain that the sequence
$(x_k)_k$ produced by the St\"ormer-Verlet 
method satisfies the same recursive relation
as the one produced by the Symplectic Euler scheme,
but with a different trigger. 
\end{remark}}

{ In order to design a discrete restart
procedure, a first natural attempt is to formulate
a discrete-time version of the Maximum Mean
Dissipation. We recall that in the continuous-time
setting this procedure consists in restarting
the evolution as soon as the quantity
\begin{equation} \label{eq:Mean_Diss_sec_discr}
r(t):=\frac{|v(t)|^2}{2t} 
\end{equation}
attains a local maximum. Hence we can restart the
discrete system \eqref{alg:cons} as soon as the
following condition is met:
\begin{equation}\label{eq:discr_mmd_n}
\frac{|v_k|^2}{k-l} > \frac{|v_{k+1}|^2}{k+1-l},
\end{equation}
where $l$ is $0$ or it is the index when the latest
restart has occurred. On the other hand, a local 
maximum can be characterized by a change of sign
of the first derivative
\[
\dot r(t) = -\frac{|v(t)|^2+2t\nabla f(x(t))\cdot
v(t)}{2t},
\]
thus we can restart the discrete evolution when
\begin{equation}\label{eq:discr_mmd_n_2}
|v_{k+1}|^2 + 2(k+1-l)\nabla f(x_{k+1})\cdot
v_{k+1}>0,
\end{equation}
where $l$ is $0$ or it is the index when the latest
restart has occurred.
We call {\it Restart-Conservative Method with 
maximum mean dissipation} (RCM-mmd-r)
the procedure given by \eqref{alg:cons} with
restart condition \eqref{eq:discr_mmd_n}.
We call {\it Restart-Conservative Method with 
differential maximum mean dissipation} (RCM-mmd-dr)
the procedure given by \eqref{alg:cons} with
restart condition \eqref{eq:discr_mmd_n_2}.
These methods are described
respectively in Algorithm~\ref{Alg_mmd_1}
and Algorithm~\ref{Alg_mmd_2}.
\begin{algorithm}
\caption{Restart-Conservative Method 
with maximum mean dissipation (RCM-mmd-r)}
\begin{algorithmic}[1]
  \STATE $x \gets x_0 - h^2 \nabla f(x_0)$
  \STATE $v \gets -h\nabla f(x_0)$
  \STATE $i \gets 1$
  \STATE $l\gets i-1$
  \WHILE{$i \leq max\_iter$}
  \STATE $i \gets i+1$
  \STATE $v' \gets v - h\nabla f(x)$
  \STATE $x' \gets x + hv'$
  \IF{$|v'|^2/(i-l)<|v|^2/(i-1-l)$}
  \STATE $x \gets x -h^2 \nabla f(x)$
  \STATE $v \gets -h \nabla f(x)$
  \STATE $l\gets i-1$
  \ELSE
  \STATE $x \gets x'$
  \STATE $v \gets v' $
  \ENDIF  
  \ENDWHILE
\end{algorithmic}
\label{Alg_mmd_1}
\end{algorithm}

\begin{algorithm}
\caption{Restart-Conservative Method 
with differential maximum mean dissipation 
(RCM-mmd-dr)}
\begin{algorithmic}[1]
  \STATE $x \gets x_0 - h^2 \nabla f(x_0)$
  \STATE $v \gets -h\nabla f(x_0)$
  \STATE $i \gets 1$
  \STATE $l\gets i-1$
  \WHILE{$i \leq max\_iter$}
  \STATE $i \gets i+1$
  \STATE $v' \gets v - h\nabla f(x)$
  \STATE $x' \gets x + hv'$
  \IF{$|v'|^2 + 2(i-l)\nabla f(x')\cdot v'>0$}
  \STATE $x \gets x -h^2 \nabla f(x)$
  \STATE $v \gets -h \nabla f(x)$
  \STATE $l\gets i-1$
  \ELSE
  \STATE $x \gets x'$
  \STATE $v \gets v -h \nabla f(x) $
  \ENDIF
  \ENDWHILE
\end{algorithmic}
\label{Alg_mmd_2}
\end{algorithm}

Now we consider an alternative restart strategy,
for which we can prove a {\it qualitative}
global convergence result. Indeed a}
 natural request for 
{ our discrete}
algorithm is that, at { each} iteration, the
decrease of the objective function is greater
or equal than the decrease achieved by the
gradient descent method with the same step.
Let $f:\R^n \to \R$ be a $C^1$-convex function
and let us define 
$z_{k+1}= x_k - h^2 \nabla f (x_k)$, then,
owing to the convexity of $f$, we have that
\[
f(z_{k+1}) \geq f(z_{k+1} + hv_k) -
\nabla f(z_{k+1} + hv_k) \cdot hv_k.
\]
Recalling that $x_{k+1}=z_{k+1} + hv_k$, we deduce
that as long as the following inequality holds
\begin{equation} \label{eq:discr_rest}
\nabla f(x_{k+1}) \cdot v_k \leq 0,
\end{equation}
then we have that
\[
f(x_k - h^2 \nabla f(x_k)) \geq f(x_{k+1}).
\]
We can use inequality \eqref{eq:discr_rest} to
design a restart criterion for the 
sequence defined in \eqref{alg:cons}: when
\eqref{eq:discr_rest} is violated, i.e.,
\[
\nabla f(x_{k+1}) \cdot v_k > 0,
\] 
then we set
\[
x_{k+1}= x_k - h^2 \nabla f(x_k), \,\,\,\,\,\,
v_{k+1}= -h\nabla f(x_k).
\]
We call this procedure 
{\it Restart-Conservative Method with
gradient restart} (RCM-grad) and
we present its implementation in 
Algorithm~\ref{Alg_grad}.
This method coincides with the one described
in \cite{TPL}.
\begin{algorithm}
\caption{Restart-Conservative Method 
with gradient restart (RCM-grad)}
\begin{algorithmic}[1]
  \STATE $x \gets x_0$
  \STATE $v \gets 0$
  \WHILE{$i \leq max\_iter$}
  \STATE $x' \gets x -h^2 \nabla f(x) +hv$
  \IF{$\nabla f(x') \cdot v > 0$}
  \STATE $x \gets x -h^2 \nabla f(x)$
  \STATE $v \gets -h \nabla f(x)$
  \ELSE
  \STATE $x \gets x'$
  \STATE $v \gets v -h \nabla f(x) $
  \ENDIF
  \STATE $i \gets i+1$
  \ENDWHILE
\end{algorithmic}
\label{Alg_grad}
\end{algorithm}

\begin{remark}
It is interesting to observe
that the discrete restart condition
\begin{equation} \label{eq:ineq_rest_discr}
\nabla f(y_k +hv_k)\cdot v_k > 0
\end{equation}
is the discrete-time analogue of the inequality
\[
\nabla f(x(t))\cdot \dot{x}(t)\geq 0,
\]
which is satisfied as soon as the kinetic energy
function $E(t)=\frac12 |\dot{x}(t)|^2$ stops
growing.
{  This fact suggests that we may also 
consider a discrete restart strategy based on the
maximization of the kinetic energy. Namely, we
can restart the evolution of the discrete
system \eqref{alg:cons} as soon as 
\begin{equation} \label{eq:disc_kin_max_n}
\frac12 |v_k|^2 > \frac12 |v_{k+1}|^2.
\end{equation}
This restarted algorithm was proposed in \cite{SI},
where the authors proved some partial results
when dealing with quadratic objectives.
}
\end{remark}
The convergence of  RCM-grad
for { strictly}
convex functions {  in $C^{1,1}_L$} 
descends
directly from the convergence of the gradient
method, as shown in the following result.

\begin{theorem} \label{thm:conv_discr}
Let $f:\R^n \to \R$ be a 
{ $C^{1,1}_L$ strictly} convex
function 
{  that admits a unique minimizer}
$x^*\in \R^n$.
Let $(x_k)_{k\geq 0} \subset \R^n$ be the
sequence produced by RCM-grad with
time-step 
{ $0<h<\frac{\sqrt2}{\sqrt{L}}$}.
{  Then the sequence converges to the
minimizer $x^*$.}
\end{theorem}
\begin{proof}
{  Owing to the construction of
RCM-grad, we have that, for every $k\geq 0$
the following inequality holds:
\[
f(x_{k+1}) -f(x_k)\leq f(x_k - h^2 \nabla f(x_k))
-f(x_k).
\]
By the fact that $f$ is a convex function in
$C^{1,1}_L$, we deduce
that $f(x+\nu)\leq 
f(x) + \langle \nabla f(x),\nu \rangle
+ \frac{L}{2}|\nu|^2$ for every $x,\nu \in \R^n$
(see, e.g., 
\cite[Theorem~2.1.5]{N18}).
This implies that
\begin{equation} \label{eq:estimate}
f(x_{k+1}) - f(x_k) \leq 
-\omega (h)|\nabla f(x_k)|^2,
\end{equation}
where
\[
\omega(h) = h^2\left(1-\frac{L}{2}h^2\right).
\]
Since we have that $\omega(h)>0$ when 
$0<h<\frac{\sqrt2}{\sqrt{L}}$, then
$f(x_{k+1})\leq f(x_k)$ for every $k\geq0$.
Moreover, we observe that the function $f$ is 
coercive, since it is assumed to be strictly convex
and to admit a minimizer. Therefore, the sequence
$(x_k)_{k\geq 0}\subset \{ x\in \R^n:\, f(x)
\leq f(x_0) \}$ is bounded.
Hence we can repeat the argument of 
\cite[Theorem~2.1.14]{N18} to deduce that
$f(x_k) \to f(x^*)$ as $k\to \infty$.
Using again the boundedness of the sequence
$(x_k)_{k\geq0}$ and the fact that $f$ admits
a unique minimizer, we obtain that 
$x_k \to x^*$ as $k\to\infty$.}  
\end{proof}

{ 
\begin{remark}
If we define  $e_k:=f(x_k)- f(x^*)$
for every $k\geq0$, and if we assume   
$f \in \mathscr{S}^{1,1}_{\mu,L}$, then, combining  
\eqref{eq:grad_vs_f_SC} with  \eqref{eq:estimate},
we obtain $e_{k+1} \le e_k (1 - 2\mu \omega(h))$ for 
$0<h<\sqrt{{2}/{L}}$. 
Choosing $\bar{h}={1}/{\sqrt{L}}$ we have
as a byproduct the linear convergence estimate
$ e_{k+1}\le e_k ( 1-\frac{\mu}{L})$.
We point out that Theorem~\ref{thm:conv_discr} and the 
previous estimate
should be understood as a {\it qualitative}
global convergence result.
Indeed, on one hand, we may deduce that the 
convergence rate of RCM-grad is {\it at least}
as fast as the convergence rate of the classical 
gradient descent. On the other hand,
this estimate of the performances of RCM-grad is
very pessimistic, as shown in the numerical
experiments of Section~\ref{sec_num_test}.
To the best of our knowledge,
Theorem~\ref{thm:conv_discr} is the first 
convergence result for an 
optimization algorithm based on
the discretization of the conservative dynamics, 
although it does not provide
the convergence rate observed in 
the numerical experiments developed in the paper. 
\end{remark}
}

\begin{subsection}{Choice of the time-step}
The proof of the convergence of RCM-grad 
holds true for any choice of the time-step
$h$ such that the gradient method with step-size
$h^2$ is convergent.
In this subsection, we provide considerations
about the choice of time-step $h$ by studying 
the one-dimensional quadratic case.
Let us fix $a>0$ and let us consider 
$f(x) = \frac12 a x^2$. 
In this case the 
sequences $(x_k)_{k\geq 0}$ and $(v_k)_{k\geq 0}$ 
are recursively defined as
\begin{equation} \label{eq:seq_quad}
\begin{cases} 
v_{k+1}= v_k  -ha\, x_k, \\
x_{k+1}= x_k +h\,v_{k+1}.
\end{cases}
\end{equation}
It is easy to check that the following discrete conservation holds:
\begin{equation*}
\frac{1}{2}v_k^2 + \frac{a}{2}x_k^2 - \frac{1}{2}ah\, x_kv_k = \frac{a}{2}x_0^2.
\end{equation*}
This implies that the sequence of points $(x_k,\,v_k)_{k\geq 0} \in \R^2$ lies on 
the following
conic curve in the $(x,v)$-plane:
\begin{equation} \label{eq:conic_seq}
\frac{1}{2}v^2 + \frac{a}{2}x^2 + \frac{1}{2}ah\, xv = c.
\end{equation}
It is natural to set $h$ such that
 the curve defined by \eqref{eq:conic_seq} is
compact. Using the characterization
 of conic curves in the plane,
we obtain that
\begin{equation} \label{eq:h_bound_1}
h<\frac{2}{ \sqrt{a}}.
\end{equation}
Another natural request is to impose that 
$|x_1|<|x_0|$ and that $x_0 x_1 \geq 0$.
Using \eqref{eq:seq_quad}, we have that 
$x_1 = (1 -ah^2)x_0$, so we impose that
$1>1 -ah^2>0$, and we deduce that
\begin{equation} \label{eq:h_bound_2}
h<\frac{1}{\sqrt{a}}.
\end{equation}
For a generic
{ $C^{1,1}_L$} convex
function $f: \R^n \to \R$, an heuristic rule for 
designing $h$ could be to use \eqref{eq:h_bound_2}, 
where $a$ is a constant that bounds from above 
{  the Lipschitz 
constant of $\nabla f$}.
\end{subsection}

\begin{subsection}{Nesterov Accelerated Gradient 
methods with restart} \label{subsec:NAG-restart}
We recall that the most efficient algorithms
for convex optimization problems
belong to the  the family of the
Nesterov Accelerated Gradient methods 
(see \cite{N83}, \cite{N18}).
We use the acronym NAG to refer to this family.
Namely, when the problem consists in minimizing a
function 
{ in $\mathscr{S}^{1,1}_{\mu,L}$
whose constant $\mu$ is known},
the most performing 
algorithm is called NAG-SC and it is defined as
\begin{equation} \label{alg:NAG-sc}
\begin{cases}
y_{k+1} = x_k -s\nabla f (x_k), \\
x_{k+1} = y_{k+1} + \frac{1 - \sqrt{\mu s}}{1 + \sqrt{\mu s}}(y_{k+1} - y_k), 
\end{cases}
\end{equation} 
with $0<s\leq \frac{1}{ L}$.

\noindent
On the other hand, when dealing with a 
non-strongly convex function 
{ in $C^1_L$},
the most suitable algorithm is 
NAG-C, whose update rule is
\begin{equation} \label{alg:NAG-c}
\begin{cases}
y_{k+1} = x_k -s\nabla f (x_k), \\
x_{k+1} = y_{k+1} + \frac{k}{k+3}(y_{k+1} - y_k), 
\end{cases}
\end{equation} 
where, as above, $0<s\leq\frac{1}{ L}$.

\noindent
In order to boost the convergence of NAG-C
via adaptive restart,
in \cite{OC} O'Donoghue and Cand\`es suggest
some schemes reproducing in a discrete form
the requirement that 
$f(x(t))$ is monotone decreasing along 
the curve $t\mapsto x(t)$, solution of an ODE
with suitable friction term. More precisely, they
proposed to restart \eqref{alg:NAG-c} as
soon as $f(y_{k+1})>f(y_k)$ ({\it function
scheme}), or as soon as 
$\nabla f(x_{k+1}) \cdot (y_{k+1} - y_k)>0$
({\it gradient scheme}). 
The intuitive idea that lies behind the 
latter scheme is to restart the evolution
when the momentum and 
the negative direction of the gradient form
an obtuse angle.
We recall that the update of $y_{k+1}$
coincides with a step of the gradient method,
namely $y_{k+1}=x_k -s\nabla f (x_k)$.
Hence we have that the step of NAG-C is given
by
\[
x_{k+1} = y_{k+1} + w_k,
\]
where
\[
w_k = \beta_k (y_{k+1} - y_{k}) \,\,\,
\mbox{and }\,\,
\beta_k= \frac{k}{k+3}.
\]
Using these facts, the gradient restart scheme
for NAG-C can be better motivated. Indeed,
as done before for the conservative algorithm,
we can impose that
each iteration of NAG-C achieves a greater 
decrease of the objective function
than a step of the gradient method:
\begin{equation}\label{eq:NAG_vs_grad}
f(y_{k+1}) \geq f(y_{k+1} + w_k).
\end{equation}
If we apply {\it verbatim} the
reasoning done { before}
about the restart of the conservative
method, then for every $C^1$-convex function
we { have that
\begin{equation*}
f(y_{k+1}) \geq f(y_{k+1} + w_k)
- \nabla f(y_{k+1}+ w_k) \cdot w_k.
\end{equation*}
Recalling that $x_{k+1}=y_{k+1}+ w_k$
and that $w_k= \beta_k(y_{k+1}-y_k)$, we deduce that
\eqref{eq:NAG_vs_grad} is satisfied
as long as
the following condition holds:
\begin{equation} \label{eq:grad_rest_NAG}
\nabla f(x_{k+1})\cdot(y_{k+1}-y_k)\leq 0.
\end{equation}
If we restart the method as soon as 
\eqref{eq:grad_rest_NAG} is violated, we recover}
the gradient restart
scheme proposed by O'Donoghue and Cand\`es
in \cite{OC}.
In conclusion, this proves that the NAG-C
with the gradient restart scheme
achieves, at {\it { each} iteration},
an effective acceleration
with respect to the classical gradient descent.
In the experiments reported in \cite{OC},
the authors show  that the 
gradient restart scheme
has better performances than the function 
restart scheme.
For these reasons, in the numerical 
tests reported in Section~\ref{sec_num_test}
we used NAG-C with gradient
restart as benchmark for the convergence
rate. From now on, we refer to this method
as  NAG-C-restart.
An important feature of NAG-C-restart
is that it is suitable for strongly convex
minimization when the strong convexity parameter of
the objective is not available, as observed 
in \cite{N13} in the framework of composite 
optimization. This aspect was further studied in
\cite{FQ19}, where the authors proved linear
convergence results for restarted accelerated
methods when the objective function
satisfies a {\it local quadratic growth condition}.
In \cite{FQ20} this was extended to accelerated
coordinate descent methods.
Finally,
we observe that the gradient restart
scheme has been recently used in \cite{K18}
in order to accelerate the convergence of the
Optimized Gradient Method introduced in \cite{K17}.

\end{subsection}

\end{section}

\begin{section}{Numerical tests} \label{sec_num_test}

In this section we describe the numerical 
experiments that we used to test the
efficiency of our method. We used
different variants of NAG as comparison:
in particular, NAG-C-restart 
(see Subsection~\ref{subsec:NAG-restart})
is the benchmark
of our numerical tests.

\begin{subsection}{Quadratic function}
\label{subsec:quad}
We considered a quadratic function $f: \R^n \to \R$ of the form
\[
f(x) = \frac12 x^TAx + b^Tx
\]
where $n=1000$, $A$ is a symmetric positive
definite matrix and $b \in \R^n$ is sampled
using $\mathcal{N}(0,1)$.
The eigenvalues of $A$ are  randomly chosen using an uniform distribution
over $[0.03, 15]$.
 The Lipschitz constant of $\nabla f$ is $
 \lambda_{\max}$, the largest eigenvalue of $A$. 
The function $f$ is $\mu$-strongly convex
for every $0< \mu \leq \lambda_{\min}$, 
where $\lambda_{\min}$ is the minimum
eigenvalue of $A$. 
For each experiment, we run the following algorithms:
\begin{itemize}
\item NAG-SC with $s= \frac{1}{\lambda_{\max}}$
and $\mu = \lambda_{\min}$. This is the sharpest possible setting of the
parameters for the given problem;
\item NAG-SC with $s= \frac{1}{\lambda_{\max}}$
and $\mu = \frac{\lambda_{\min}}{3}$. This simulates an underestimation of the
strongly-convexity constant;
\item NAG-C-restart
 with
$s= \frac{1}{\lambda_{\max}}$;
\item RCM-grad with
 $h=\frac{1}{\sqrt{\lambda_{\max}}}$;
 \item RCM-mmd-dr
 with
 $h=\frac{1}{\sqrt{\lambda_{\max}}}$;
\item RCM-mmd-r
 with
 $h=\frac{1}{\sqrt{\lambda_{\max}}}$;
 \item RCM-kin
 with
 $h=\frac{1}{\sqrt{\lambda_{\max}}}$.
\end{itemize}
The results are described in Figure
\ref{fig:quadratic}.
This test shows that, among the Restart-Conservative
Methods, RCM-grad and RCM-mmd-dr are
the most performing.
Moreover, RCM-grad and RCM-mmd-dr achieve a faster
convergence rate than NAG-SC-2 when
a sharp estimate of the strongly-convexity constant 
is not available.
We also observe that both RCM-grad and
RCM-mmd-dr
have on average slightly better  performances 
than NAG-C-restart.
However, when the strongly-convexity constant is 
known, NAG-SC has better performances
than the other algorithms.
Finally,
RCM-mmd-r and RCM-kin
have a slower convergence rate 
than other Restart-Conservative methods.
\begin{center}
\begin{figure}
\includegraphics[width=5.9cm]{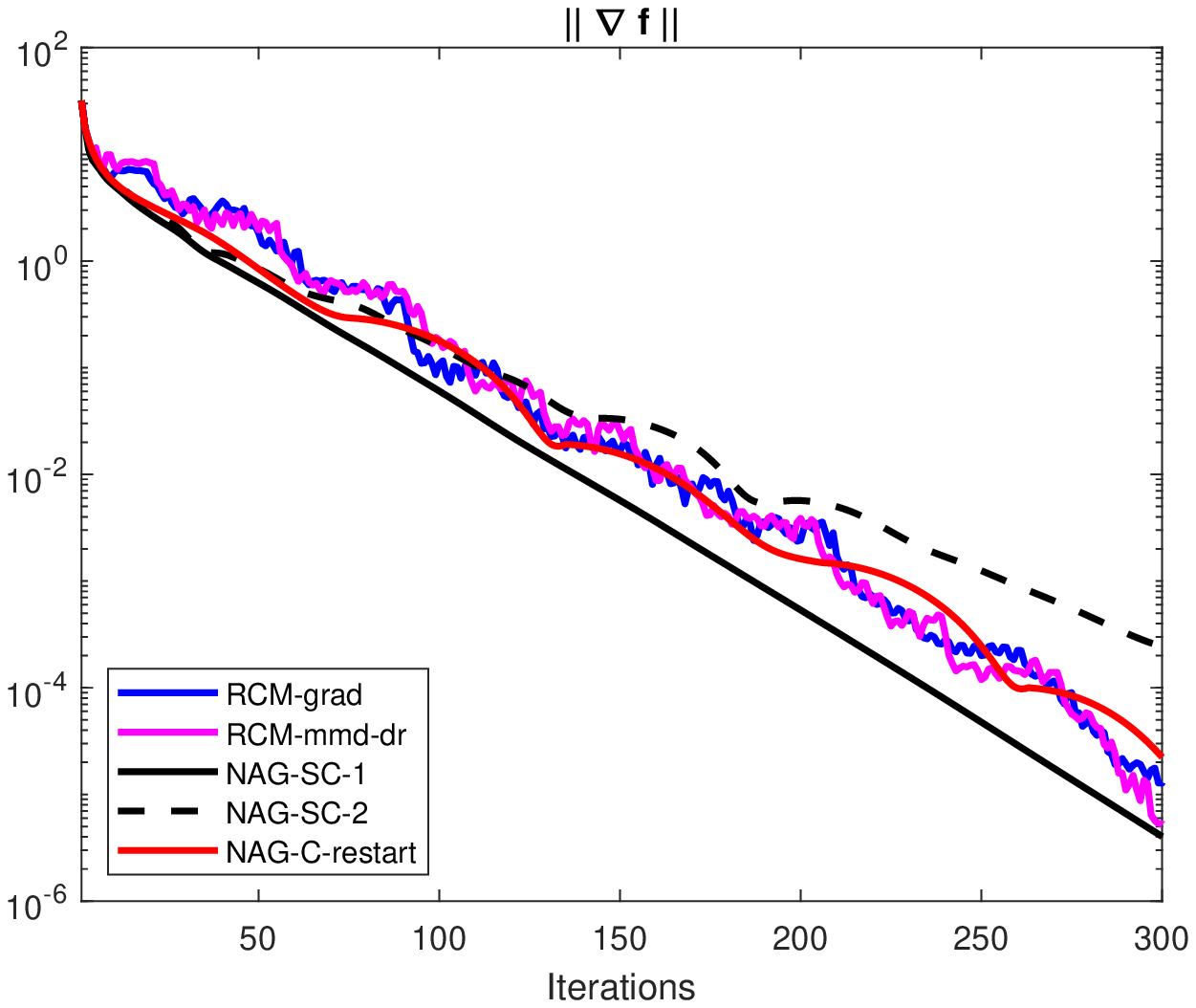}
\includegraphics[width=5.9cm]{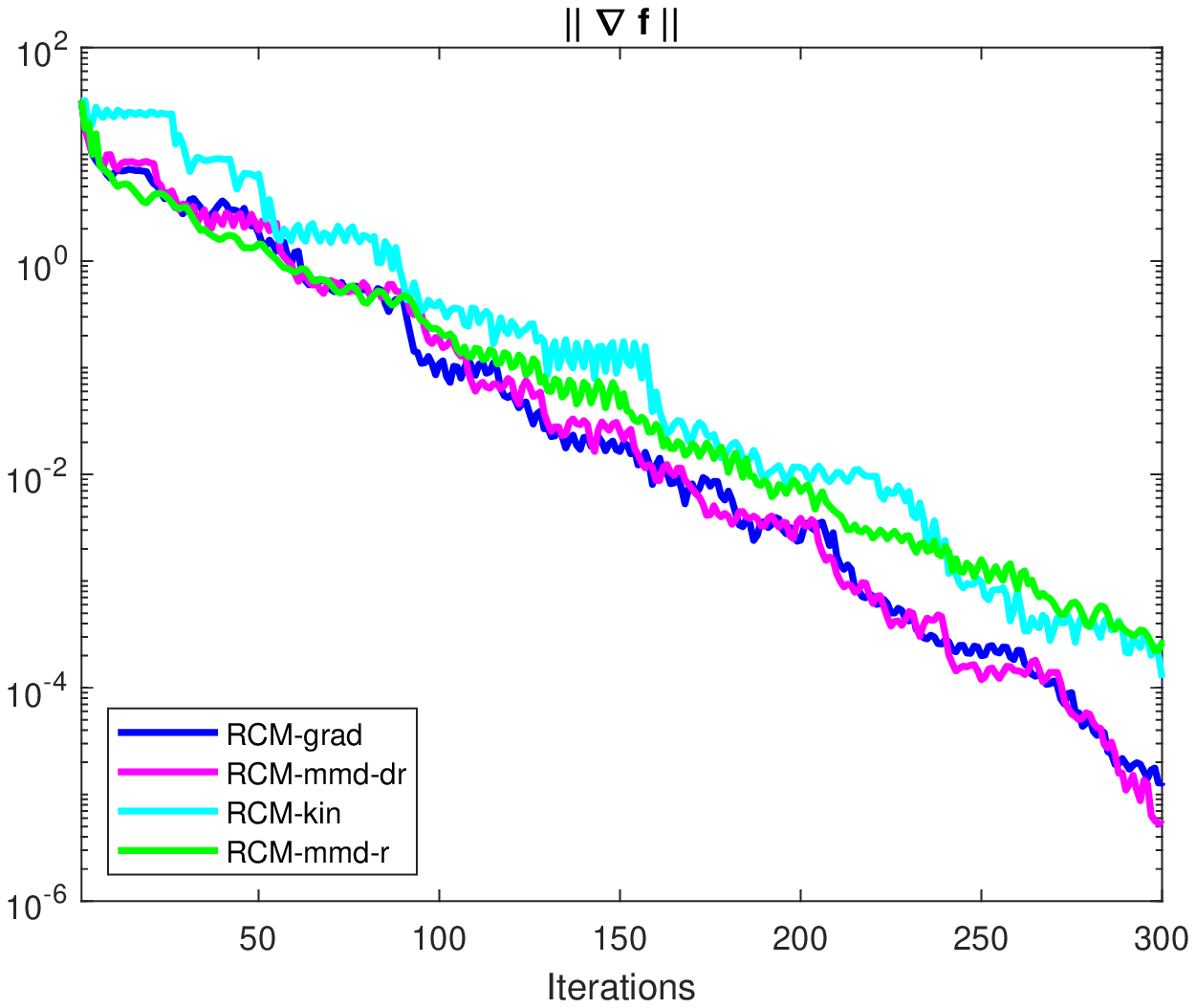}\\
\includegraphics[width=5.9cm]{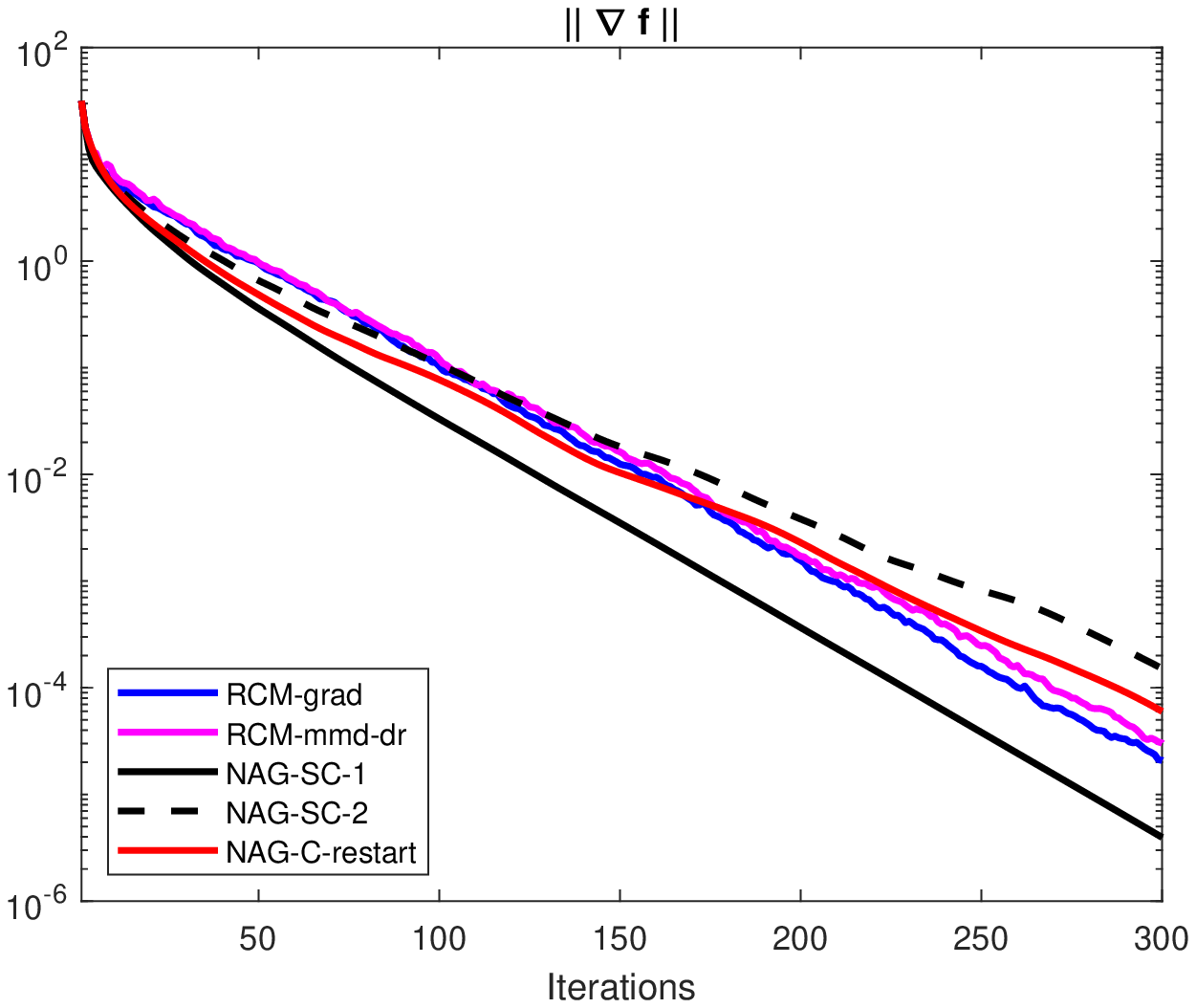}
\includegraphics[width=5.9cm]{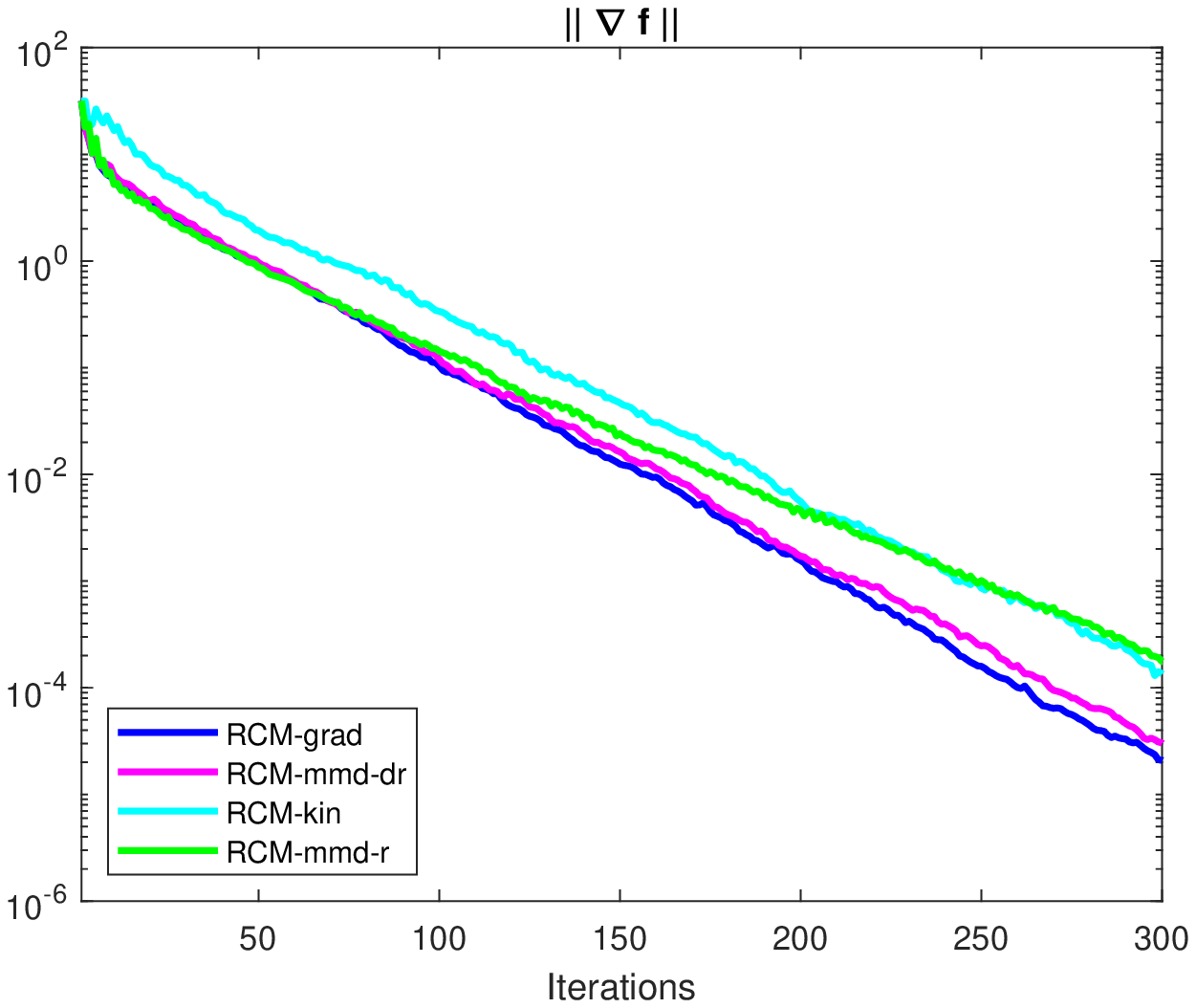}
\caption{Quadratic case. 
At the top we report the result of a single experiment,
at the bottom the average over 50 repetitions of the
experiment.
The plots at
left-hand side shows the
decay of the objective function achieved by
RCM-grad (blue), RCM-mmd-dr (magenta),
 NAG-SC with exact
strongly-convexity constant (NAG-SC-1, black), 
NAG-SC with underestimated
strongly-convexity constant (NAG-SC-2, dashed), and
NAG-C-restart (red).
At right-hand side we compare the convergence rate
of the Restart-Conservative method
 with different restart schemes.
We observe that on average
 RCM-grad and RCM-mmd-dr
 have a slightly better performances  than the benchmark NAG-C-restart.}
\label{fig:quadratic}
\end{figure}
\end{center}
\end{subsection}

\begin{subsection}{Logistic regression}
\label{subsec:logistic}
We considered a typical logistic regression
problem.
First of all, 
we randomly generated the vector 
$x_0 \in \R^n$
using $\mathcal{N}(0, 0.01)$. Then we
independently
sampled the entries of the vector 
$y =\left( y_1, \ldots , y_m \right)^T
\in \{0, 1 \}^m$ using the law
$$\mathbb{P}(Y_i = 1) 
= \frac{1}{1+e^{-a_i^Tx_0}}, $$
where $A = \left( a_1, \ldots , a_n \right)$
was a $n \times m$ matrix with i.i.d. entries
generated with the $\mathcal{N}(0,1)$
distribution. Supposing that $y$ and $A$
were known, we tried to recover $x_0$
using the log-likelyhood maximization.
This is equivalent to the
minimization of the function
\begin{equation} \label{eq:logistic}
f(x) = \sum_{i=1}^m
\left( 
(1-y_i) a_i^Tx + \log \left( 1+ e^{-a_i^Tx}  
\right)
\right)
\end{equation}
We set $n=100$ and $m=500$.
Let $L$
 be the Lipschitz constant of the function
$\nabla f$.
We recall
that function \eqref{eq:logistic} is
convex but not strongly convex.
We minimized
the right-hand-side of \eqref{eq:logistic}
using the following algorithms:
\begin{itemize}
\item Classical gradient descent method with
step-size $s = \frac{1}{L}$;
\item NAG-C-restart and
$s= \frac{1}{L}$;
\item RCM-grad with $h=\frac{1}{\sqrt{L}}$;
\item RCM-mmd-dr with $h=\frac{1}{\sqrt{L}}$;
\item RCM-mmd-r with $h=\frac{1}{\sqrt{{ L}}}$;
\item RCM-kin with $h=\frac{1}{\sqrt{{ L}}}$;
\end{itemize}
The results of the experiment are presented in 
Figure \ref{fig:Logistic}.
We observe that the most performing methods
are RCM-grad and RCM-mmd-dr  and they show  
similar behavior in both single and  average runs.
Moreover, RCM-grad and RCM-mmd-dr   have  faster convergence rates than 
NAG-C-restart. Among the RCM methods the  RCM-kin and RCM-mmd-r are the  slowest.
\begin{center}
\begin{figure}
\includegraphics[width=5.9cm]{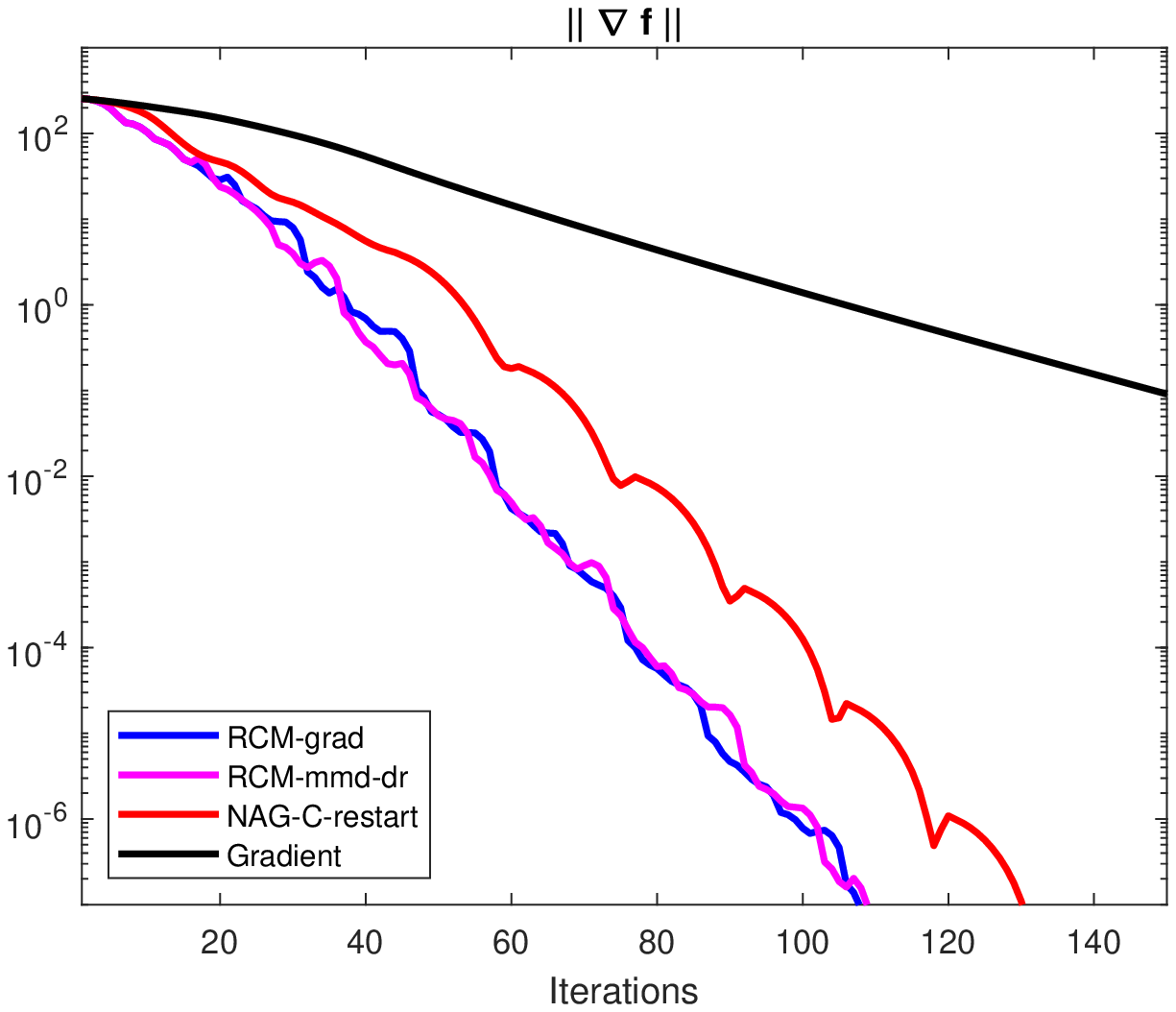}
\includegraphics[width=5.9cm]{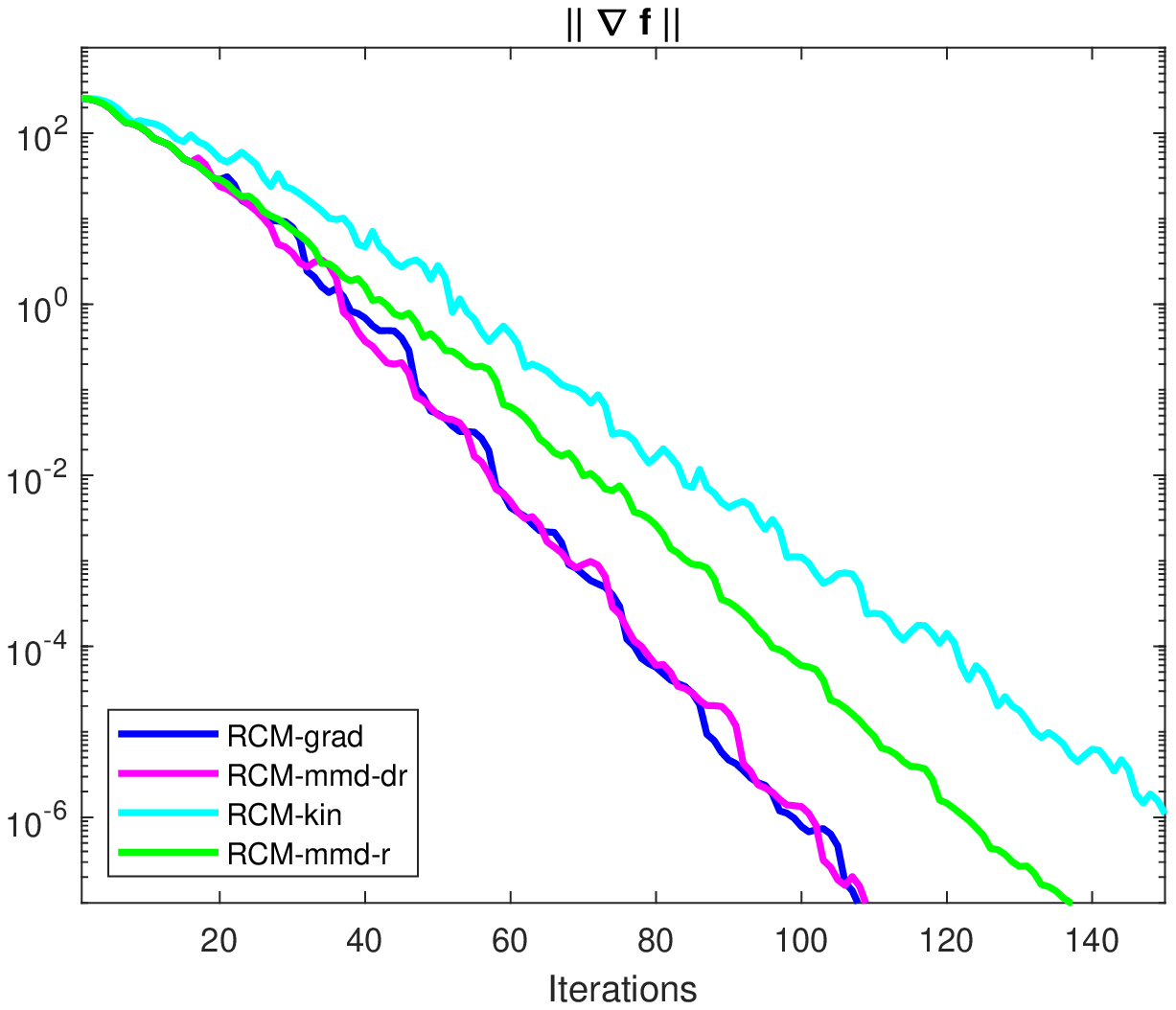}\\
\includegraphics[width=5.9cm]{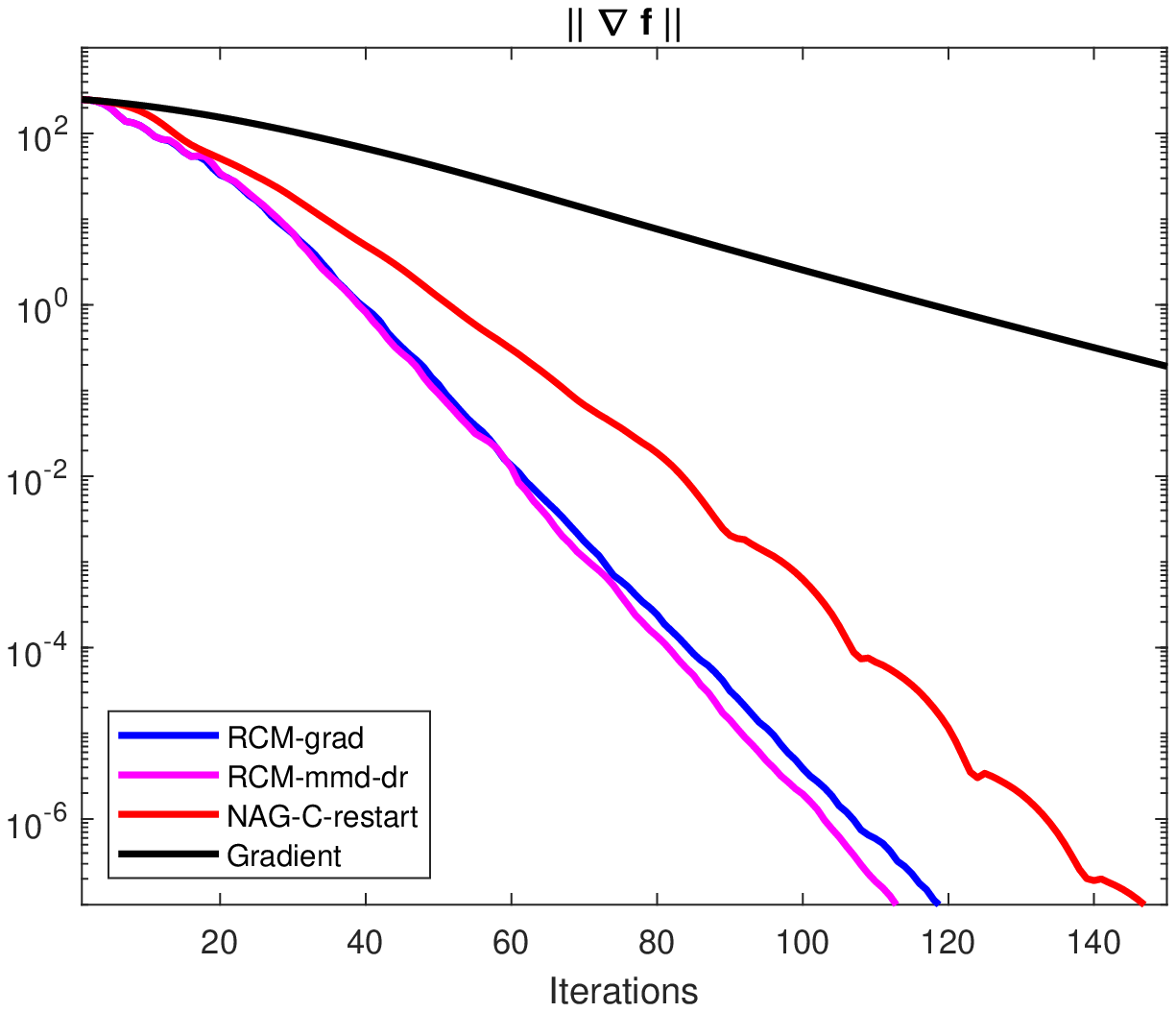}
\includegraphics[width=5.9cm]{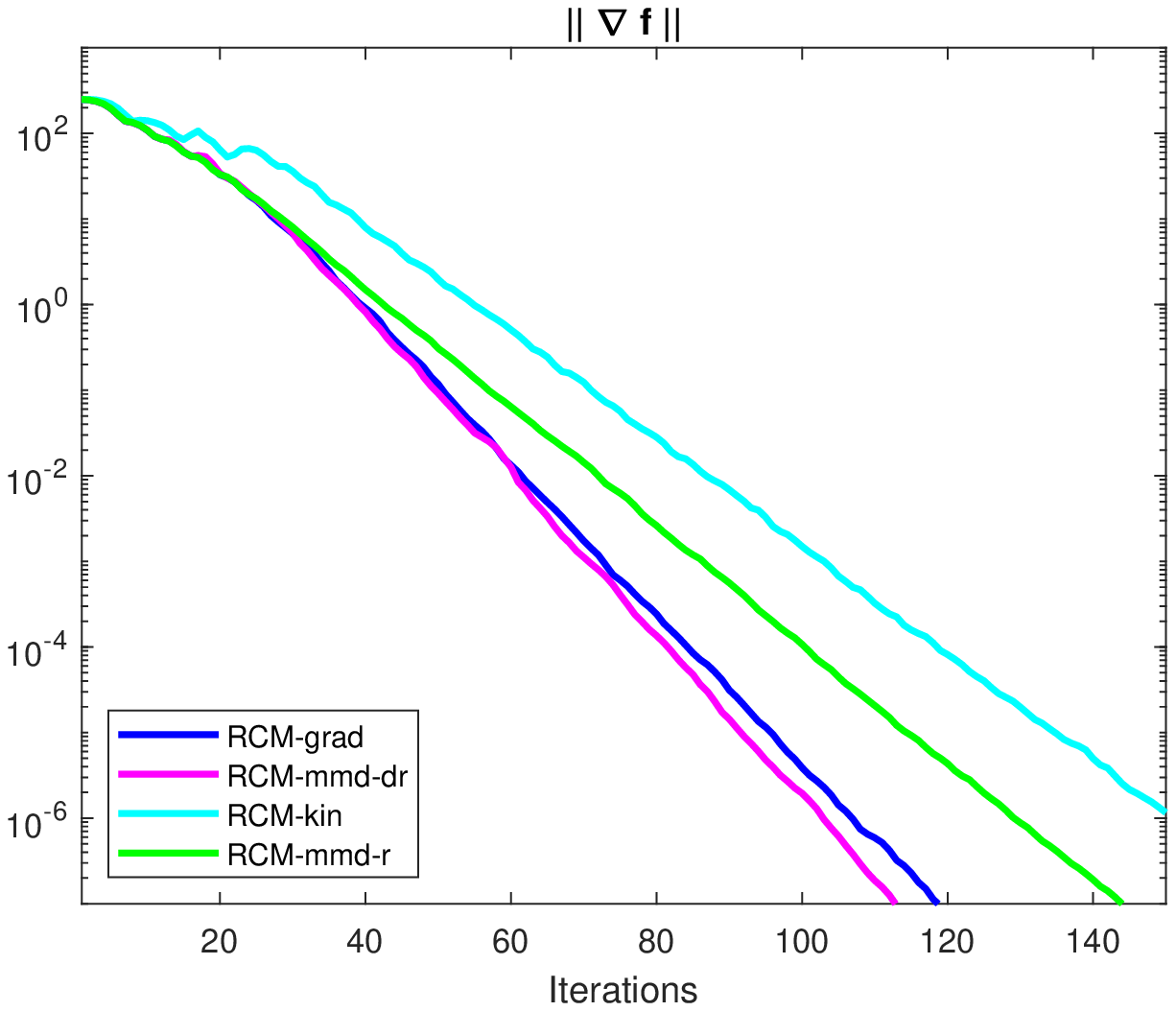}

\caption{Logistic regression. 
At the top we report the result of a single experiment,
at the bottom the average over 50 repetitions of the
experiment.
The plots
at left-hand side shows
the decay of the norm of the gradient of the 
objective function achieved 
by RCM-grad (blue), RCM-mmd (magenta), 
NAG-C-restart (red), and
the classical gradient descent (black).
At right-hand side we compare the convergence rate
of RCM with different restart schemes.
RCM-grad and RCM-mmd-dr seem 
to have faster convergence rate
than the benchmark NAG-C-restart.}
\label{fig:Logistic}
\end{figure}
\end{center}
\end{subsection}

\begin{subsection}{LogSumExp}
\label{subsec:logsumexp}
We considered the
non-strongly convex function $f:\R^n \to \R$
defined as 
\begin{equation} \label{eq:LogSumExp}
f(x) = \rho \log \left(
\sum_{i=1}^m\exp \left( 
\frac{a_i^Tx - b_i }{\rho}\right)  
\right),
\end{equation}
where $A = (a_1, \ldots, a_m)$ was a $n\times m$
matrix whose entries were independently
generated using the normal distribution 
$\mathcal{N}(0,1)$. The vector $b\in \R^m$
was sampled using $\mathcal{N}(0,1)$. We set
$n=50$, $m=200$ and $\rho = 1$.
Let ${ L}$ be the Lipschitz constant
of the function
$\nabla f$.
We minimized $f$
using the following algorithms:
\begin{itemize}
\item Classical gradient descent method with
step-size $s = \frac{1}{{ L}}$;
\item NAG-C-restart with
$s= \frac{1}{{ L}}$;
\item RCM-grad with  
$h=\frac{1}{\sqrt{{ L}}}$;
\item RCM-mmd-dr with  
$h=\frac{1}{\sqrt{{ L}}}$;
\item RCM-mmd-r with  
$h=\frac{1}{\sqrt{{ L}}}$;
\item RCM-kin with  
$h=\frac{1}{\sqrt{{ L}}}$;
\end{itemize}
The results are shown in Figure \ref{fig:LogSumExp}. 
 We observe that in the presented  single run the RCM-grad 
shows the best performance. 
In the  average
RCM-grad and RCM-mmd-dr 
exhibit very similar behaviors
 and have a slightly 
better performances than the benchmark
NAG-C-restart.
\begin{center}
\begin{figure}
\includegraphics[width=5.9cm]{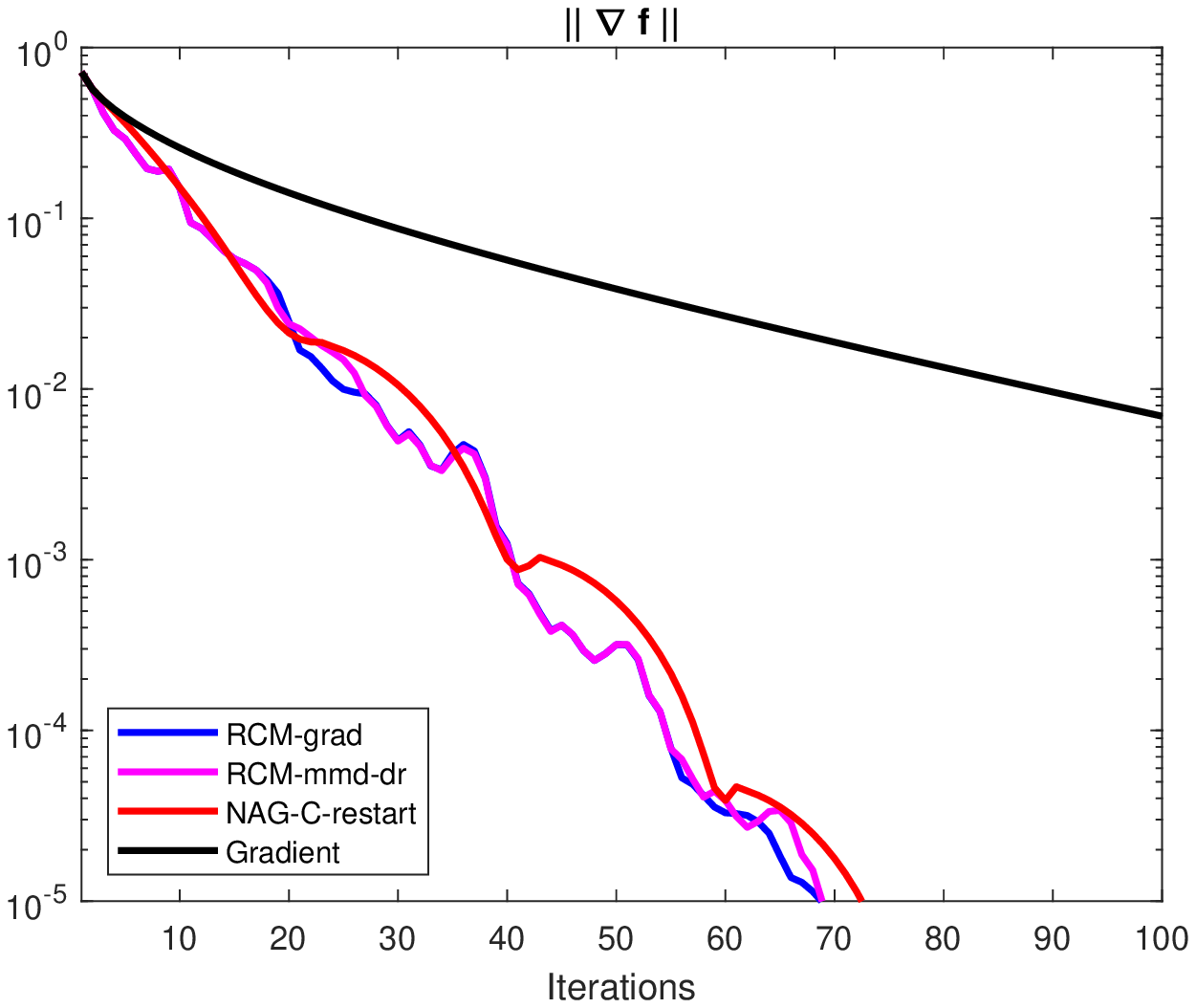}
\includegraphics[width=5.9cm]{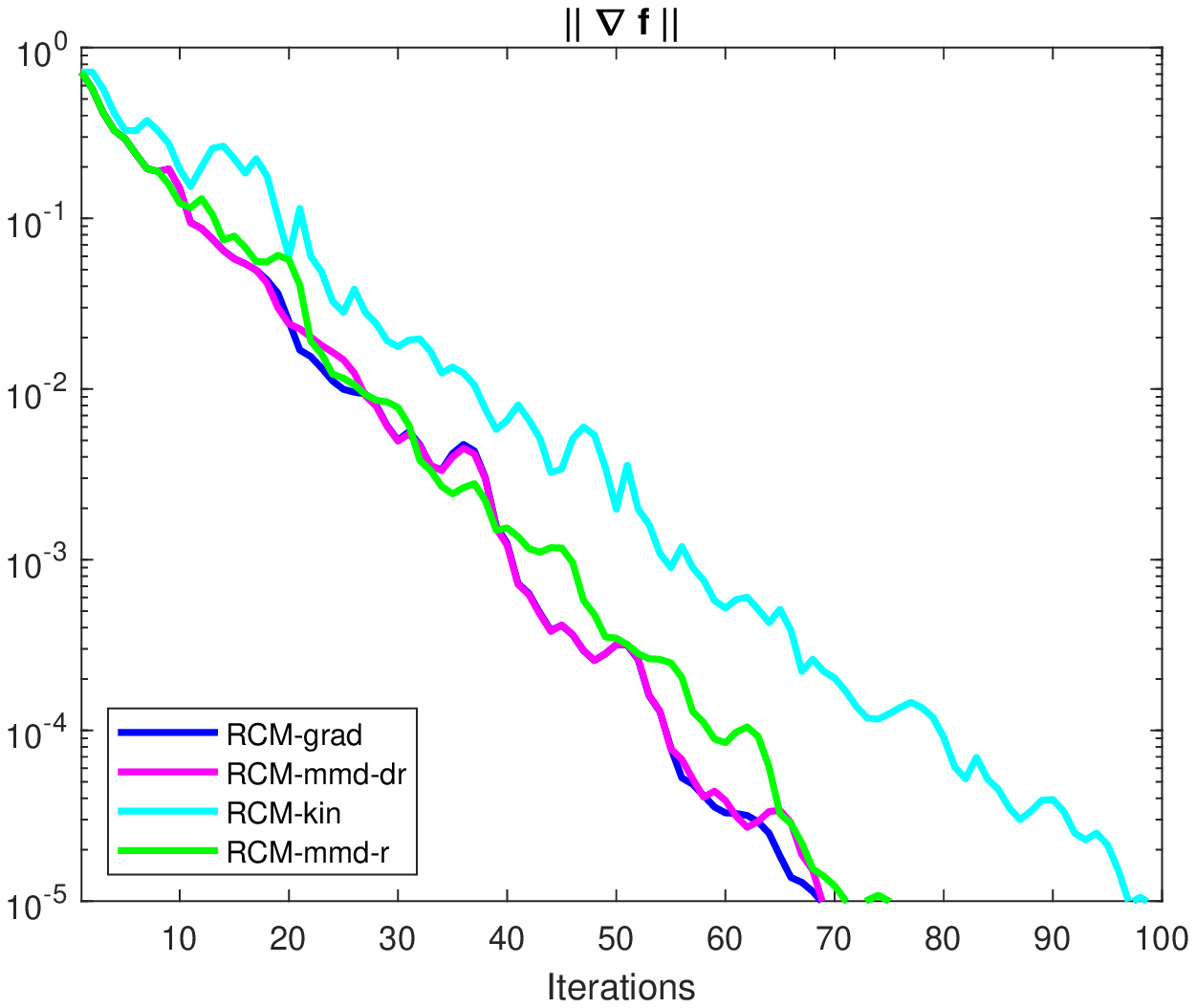}\\
\includegraphics[width=5.9cm]{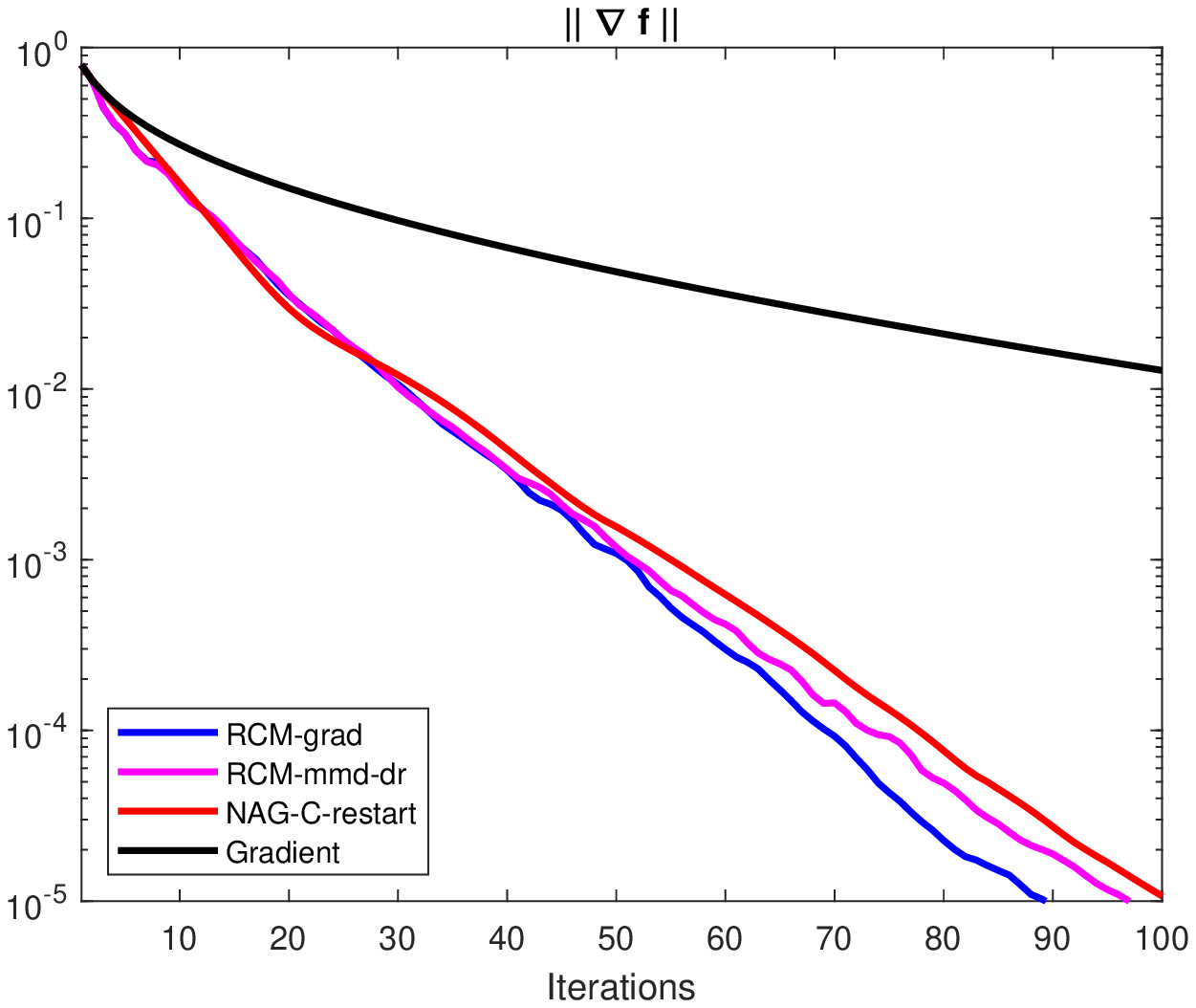}
\includegraphics[width=5.9cm]{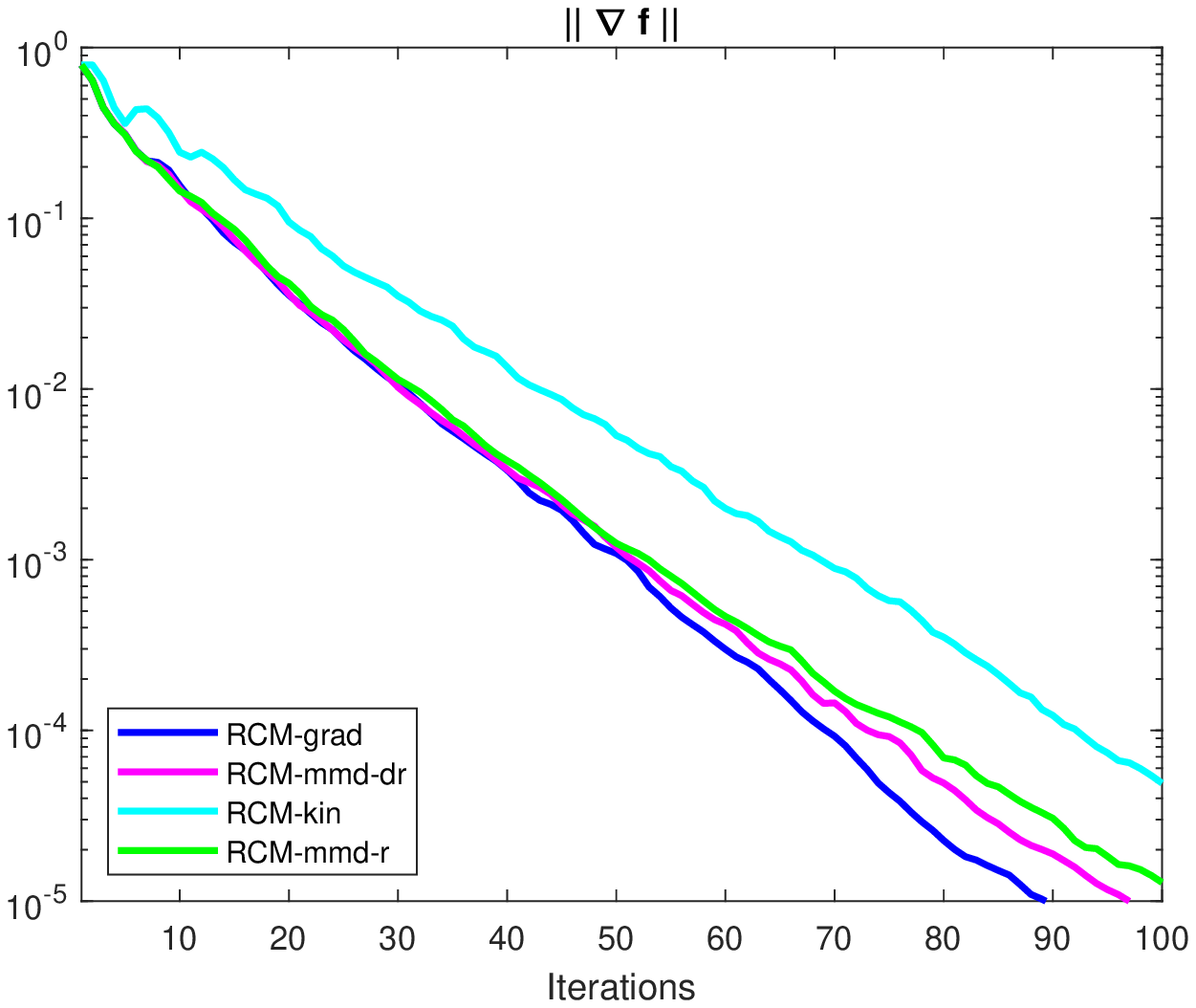}

\caption{LogSumExp function. 
At the top we report the result 
of a single experiment,
at the bottom the average over 50 repetitions of the
experiment.
The plots
at left-hand side shows
the decay of the norm of the gradient of the 
objective function achieved 
by RCM-grad (blue), RCM-mmd (magenta), 
NAG-C-restart (red), and
the classical gradient descent (black).
At right-hand side we compare the convergence rate
of RCM with different restart schemes. RCM-grad
and RCM-mmd-dr
 seem to be faster
than the benchmark NAG-C-restart.}
\label{fig:LogSumExp}
\end{figure}
\end{center}
\end{subsection}

We want to conclude this section with some
considerations about the non-smooth case. 
We try to give heuristic ideas 
to generalize our method to the
minimization of {\it composite functions}
(see, for example, \cite{N13} for an introduction
to the subject). Namely, we consider functions
$f:\R^n \to \R$ of the form
\[
f(x) = g(x) + \Psi(x),
\]
where $g:\R^n \to \R$ is a smooth convex function 
and $\Psi:\R^n \to \R$ is a Lipschitz-continuous
convex function. The main obstruction to the
direct application of RCM for the minimization
of $f$ is due to the fact that, in general,
the gradient $\nabla f$ may not be well-defined.
In order to avoid this inconvenient,
we introduce the map $\partial^-f:\R^n \to \R^n$
defined as follows:
\begin{equation} \label{eq:gen_grad}
\partial^-f(x) = \mbox{argmin} \left\{ 
|| v ||_2: \, v \in \partial f(x)
 \right\},
\end{equation}
where $\partial f(x) \subset \R^n$ is the 
sub-differential of $f$ at the point $x$, and
$|| \cdot ||_2$ denotes the Euclidean norm.
The good definition of the map $\partial^-f$
descends from general properties of convex
functions (see, for example, the textbooks 
\cite{HL}, \cite{R97}). Hence, the first modification
consists in replacing $\nabla f$ with 
$\partial^-f$.

The second modification to the original RCM
 is suggested by  
physical intuition. Let us imagine that a 
small massive ball subject to the gravity force
is constrained to move on the graph of the
function $f$. The graph is {\it sharp-shaped} in
correspondence of the non-differentiability 
points of the function $f$. 
If a physical ball crosses
these regions, we expect a loss of kinetic energy
due to the inelastic collision between the ball
and the sharp surface of the graph. 
Then, for example, we can reset the velocity
equal to zero whenever the sequence crosses a 
non-differentiability region.
This intuition can be motivated by the fact that
the quantity $\partial^-f$ usually has 
sudden variation in correspondence of
non-differentiability points of $f$.
Hence, when we cross these regions,
the information carried by the momentum
can be of little use, if not misleading.

From now on, we suppose that 
$\Psi : \R^n \to \R$ has the form:
\[
\Psi(x) = \sum_{i=1}^n|x_i|.
\]
For this choice of $\Psi$, we propose 
in Algorithm \ref{Algoritmo composite}
a variant of RCM. 
\begin{algorithm}
\caption{Restart-Conservative Method for $\ell	^1$-composite   optimization (RCM-COMP-grad)}
\begin{algorithmic}[1]
  \STATE $x \gets x_0$
  \STATE $v \gets 0$
  \WHILE{$i \leq max\_iter$}
  \STATE $x' \gets x -h^2 \partial^- f(x) +hv$
  \IF{$\partial^- f(x') \cdot v > 0$}
  \STATE $x' \gets x -h^2 \partial^- f(x)$
  \STATE $v \gets -h \partial^- f(x)$
  \ELSE
  \STATE $v \gets v -h \partial^- f(x')$
  \ENDIF
  \FOR{$j=1,\ldots , n$}
  \IF{$x'_j x_j<0$}
  \STATE $x'_j \gets 0$
  \STATE $v \gets 0$
  \ENDIF
  \ENDFOR
  \STATE $x \gets x'$
  \STATE $i \gets i+1$
  \ENDWHILE
\end{algorithmic}
\label{Algoritmo composite}
\end{algorithm}

\noindent
In the lines 11--16 of 
Algorithm~\ref{Algoritmo composite}
we check if the sequence has crossed
the set where the function $f$ is not
differentiable, i.e., the set 
$\{ x\in \R^n: \, x_1\cdots x_n =0 \}$.
If it has, we reset the velocity equal to $0$.
As done for RCM-grad, we can 
replace the gradient restart criterion at
line 5 with  the alternative restart procedures
described in 
Section~\ref{sec:discr}.
Similarly as before,
we call {RCM-COMP-kin}, {RCM-COMP-mmd-r}
and {RCM-COMP-mmd-dr} the methods obtained using
 the alternative restart criteria.
We just recall that, in the case of RCM-mmd-dr,
 in \eqref{eq:discr_mmd_n_2}
we need to replace $\nabla f (x_{k+1})$ with  
$\partial^- f (x_{k+1})$.

For the experiments concerning the 
$\ell^1$-composite optimization, we use
as benchmark
the restarted version 
of FISTA proposed in \cite{OC}:
as done for the NAG-C, in their paper
O'Donoghue and Cand\`es proposed an
adaptive restart procedure 
to accelerate the convergence of FISTA.
We refer to this algorithm as FISTA-restart.
We recall that FISTA
was originally introduced in \cite{BT}.

\begin{subsection}{Quadratic with 
$\ell^1$-regularization}
We considered the function $f:\R^n \to \R$ defined
as
\begin{equation}\label{eq:quad_l1}
f(x) = \frac12 x^TAx + b^Tx 
+ \gamma \sum_{i=1}^n|x_i|,
\end{equation}
where $A \in \R^{n\times n}$ and $b\in \R^n$
 were constructed as 
in Subsection~\ref{subsec:quad}.
We set $\gamma=\frac14 || b ||_{\infty}$,
in order to guarantee that the minimizer 
is not the origin. Let $\lambda_{\max}$ be
the greatest eigenvalue of $A$.
We minimized \eqref{eq:quad_l1} using the following
algorithms:
\begin{itemize}
\item FISTA with step-size $s=\frac{1}{\lambda_{\max}}$;
\item FISTA-restart with step-size $s=\frac{1}{\lambda_{\max}}$;
\item RCM-COMP-grad with step-size $h=\frac{1}{\sqrt{\lambda_{\max}}}$;
\item RCM-COMP-mmd-dr with step-size $h=\frac{1}{\sqrt{\lambda_{\max}}}$;
\item RCM-COMP-mmd-r with step-size $h=\frac{1}{\sqrt{\lambda_{\max}}}$;
\item RCM-COMP-kin with step-size $h=\frac{1}{\sqrt{\lambda_{\max}}}$.
\end{itemize}
The results are shown in Figure~\ref{fig:quad+l1}.
We measured the convergence rate by considering
the decay of $||\partial^- f||$ along the sequences
generated by the methods.
We observe that in this problem the most performing
algorithm is 
the benchmark FISTA-restart
and the worst is the original FISTA  without restart.  
Among the RCM algorithms  RCM-mmd-r and RCM-mmd-dr exhibit similar convergence rate  while RCM -kin  is the slowest.
Finally, RCM-COMP-grad shows an asymptotic
 convergence rate very similar to
 FISTA-restart.

\begin{center}
\begin{figure}
\includegraphics[width=5.9cm]{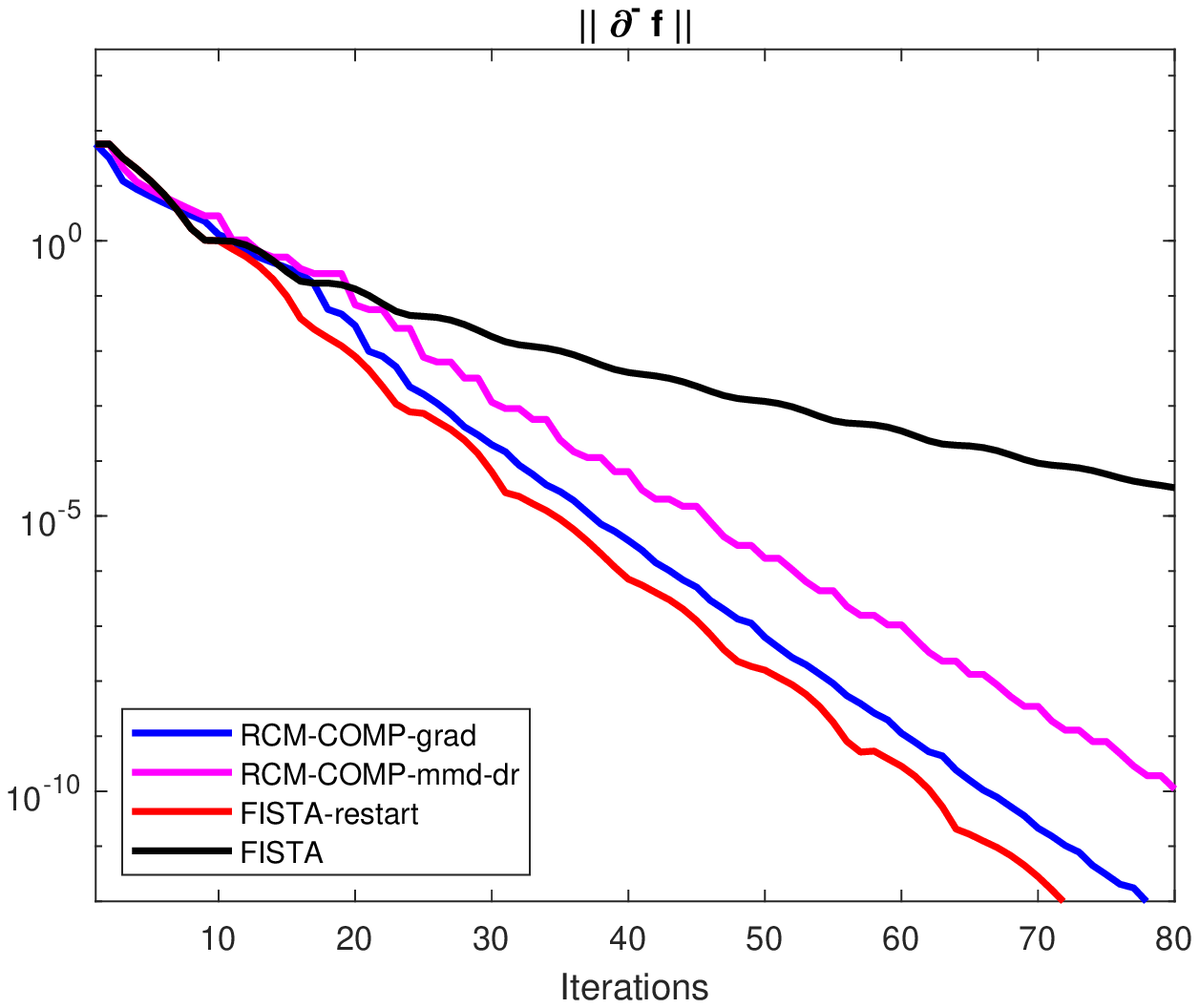}
\includegraphics[width=5.9cm]{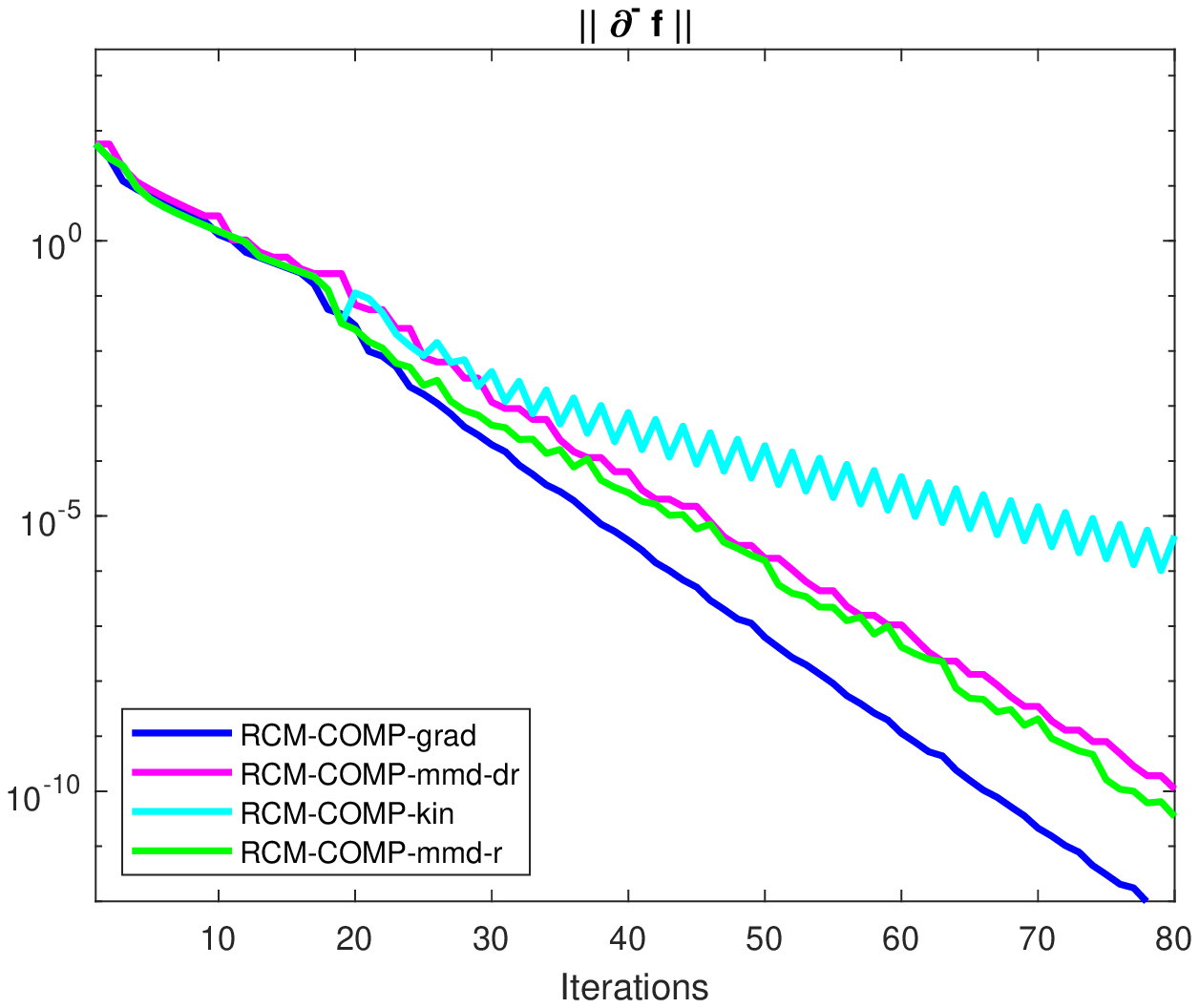}\\
\includegraphics[width=5.9cm]{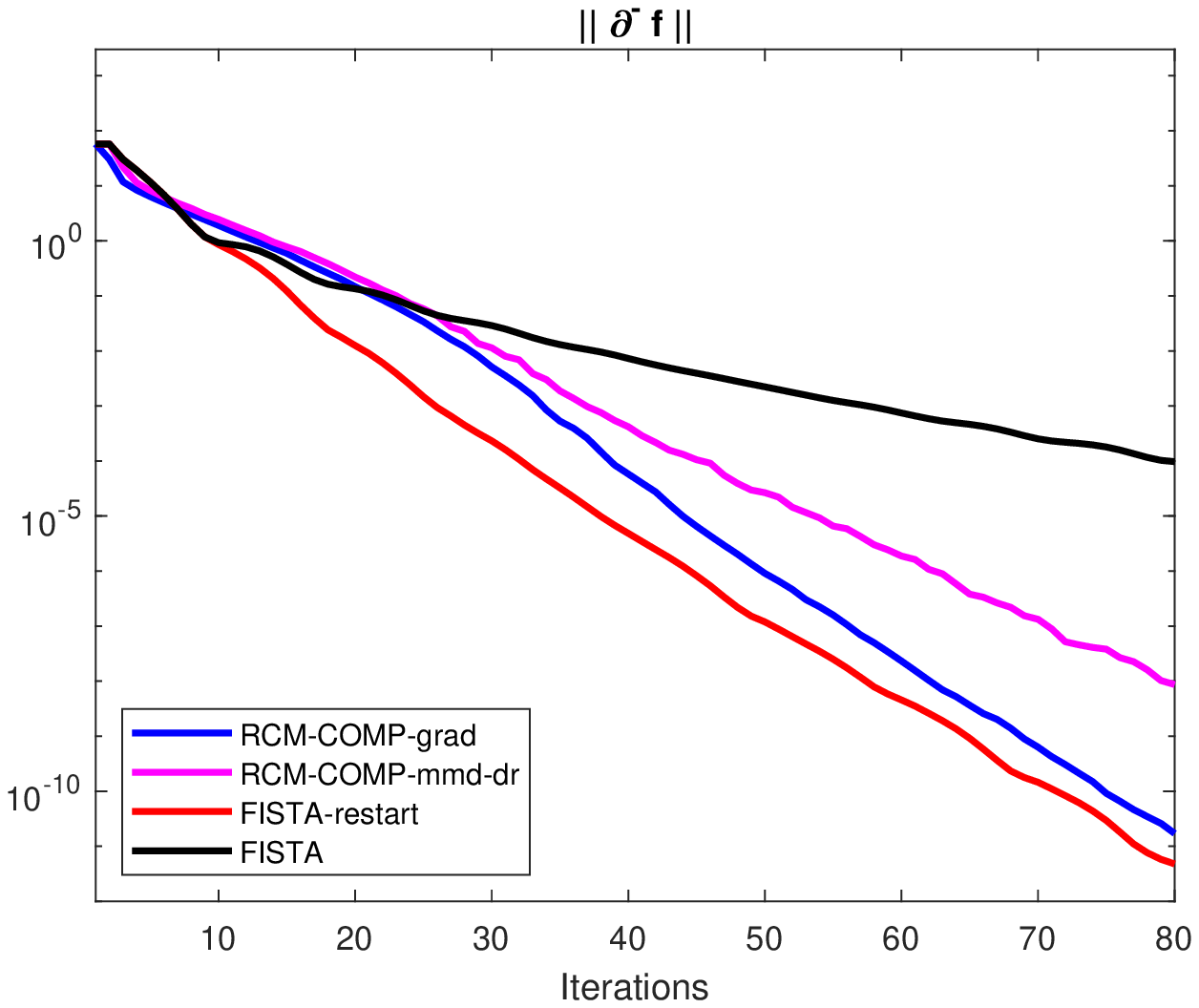}
\includegraphics[width=5.9cm]{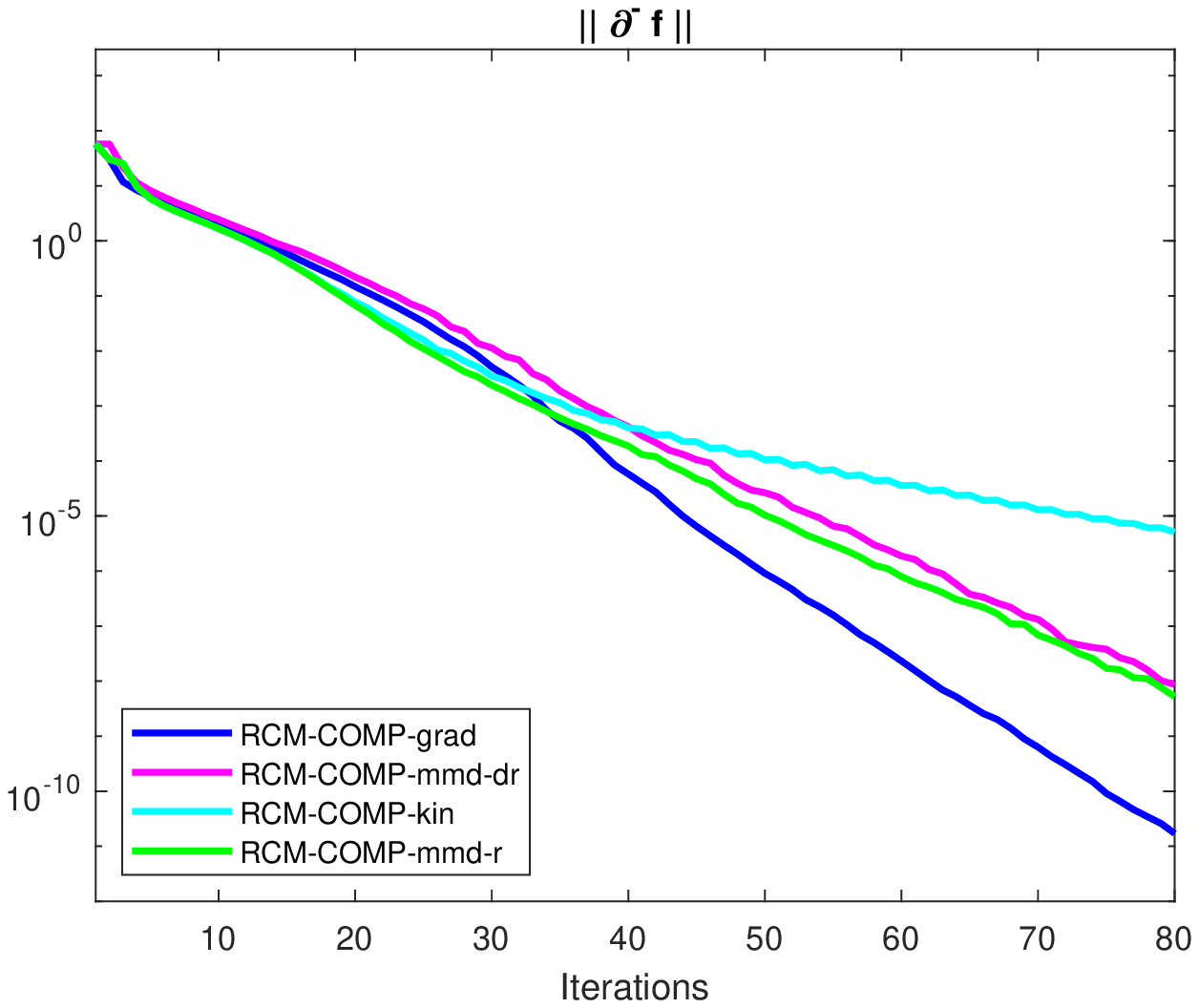}
\caption{Quadratic function 
with $\ell^1$-regularization.
At the top we report the result of a single experiment,
at the bottom the average over 
100 repetitions of the experiment.
The plots at left-hand side shows
the decay of $|| \partial^-f ||$
 achieved  by
RCM-COMP-grad (blue), RCM-COMP-mmd-dr (magenta), 
FISTA-restart (red), and
FISTA (black).
At right-hand side we compare the convergence rate
of the  Restart-Conservative method
 with different restart schemes. 
The benchmark method FISTA-restart has the best 
performances. Among the Restart-Conservative family,
RCM-COMP-grad shows the fastest convergence rate.}
\label{fig:quad+l1}
\end{figure}
\end{center}

\end{subsection}

\begin{subsection}{Logistic with 
$\ell^1$-regularization}
We considered the function $g:\R^n \to \R$ defined
as
\begin{equation*}
g(x) = \sum_{i=1}^m
\left( 
(1-y_i) a_i^Tx + \log \left( 1+ e^{-a_i^Tx}  
\right)
\right),
\end{equation*}
where $A = (a_1,\ldots,a_m) \in \R^{n\times m}$ and 
$y=(y_1,\ldots,y_m)^T \in \R^m$ 
 were constructed as 
in Subsection~\ref{subsec:logistic}.
We studied the function $f:\R^n \to \R$ defined as
\begin{equation}\label{eq:logistic+l1}
f(x) =  g(x) + \gamma \sum_{i=1}^n|x_i|.
\end{equation}
We set $\gamma=\frac12 || \nabla g(0) ||_{\infty}$,
in order to guarantee that the minimizer of $f$ 
is not the origin. 
Let ${ L}$ be the Lipschitz constant
of the function 
$\nabla g$.
We minimized \eqref{eq:logistic+l1}
 using the following algorithms:
\begin{itemize}
\item FISTA with step-size $s=\frac{1}{{ L}}$;
\item FISTA-restart with step-size $s=\frac{1}{{ L}}$;
\item RCM-COMP-grad with step-size $h=\frac{1}{\sqrt{{ L}}}$;
\item RCM-COMP-mmd-dr with step-size $h=\frac{1}{\sqrt{{ L}}}$;
\item RCM-COMP-mmd-r with step-size $h=\frac{1}{\sqrt{{ L}}}$;
\item RCM-COMP-kin with step-size $h=\frac{1}{\sqrt{{ L}}}$.
\end{itemize}
The results are shown in 
Figure~\ref{fig:logistic+l1}.
We measured the convergence rate by considering
the decay of $||\partial^- f||$ along the sequences
generated by the methods.
We recall that in this test  the smooth part of the objective function
 is a non-strongly convex function.  
We observe   that 
RCM-COMP-grad and   RCM-COMP-mmd-r 
exhibit the best  performances    while the original 
FISTA  is the worst performing method.
In this case  RCM-COMP-mmd-dr
shows on average performances very
close to the benchmark FISTA-restart and both
 exhibit the  same convergence rate.
\begin{center}
\begin{figure}
\includegraphics[width=5.9cm]{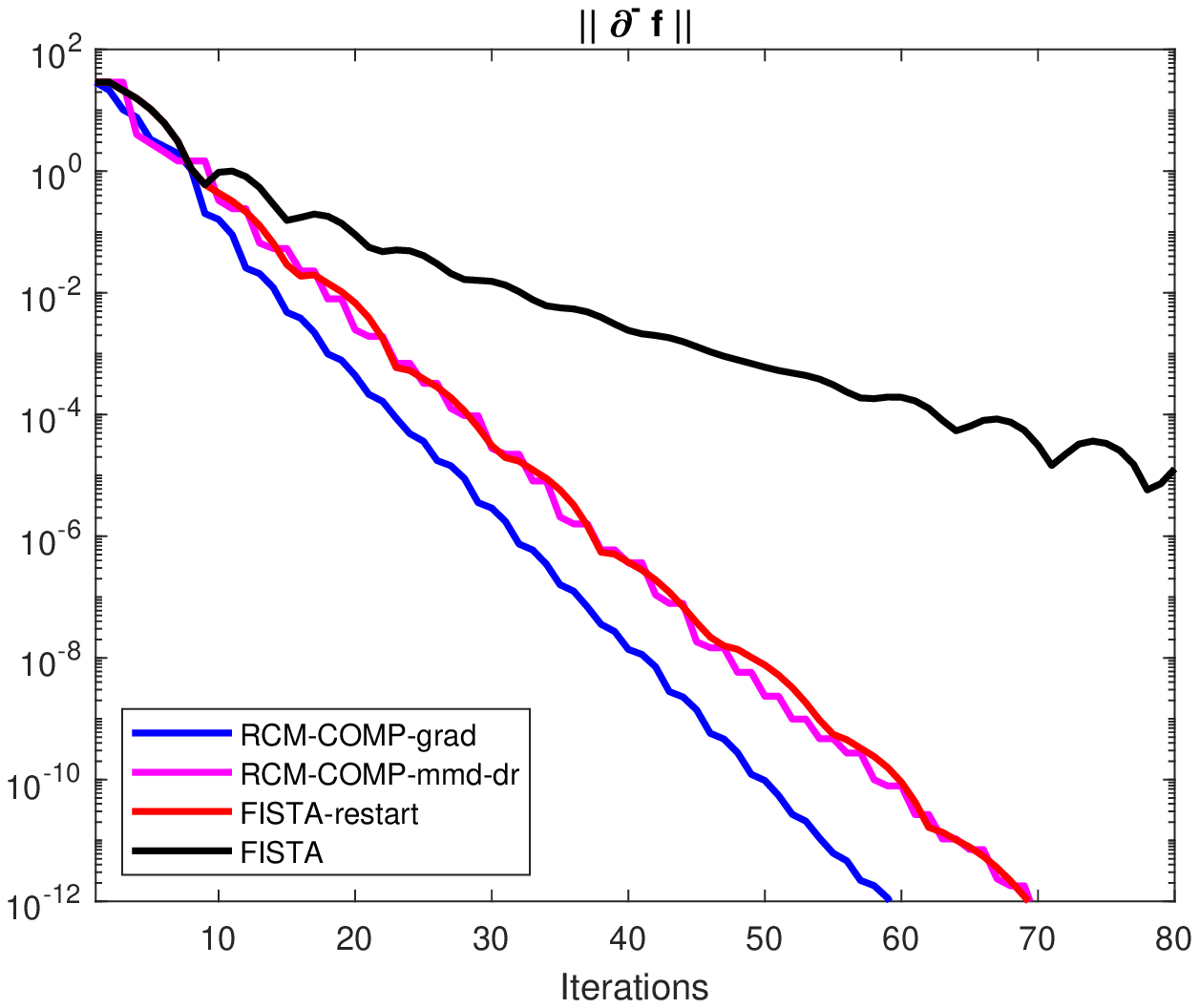}
\includegraphics[width=5.9cm]{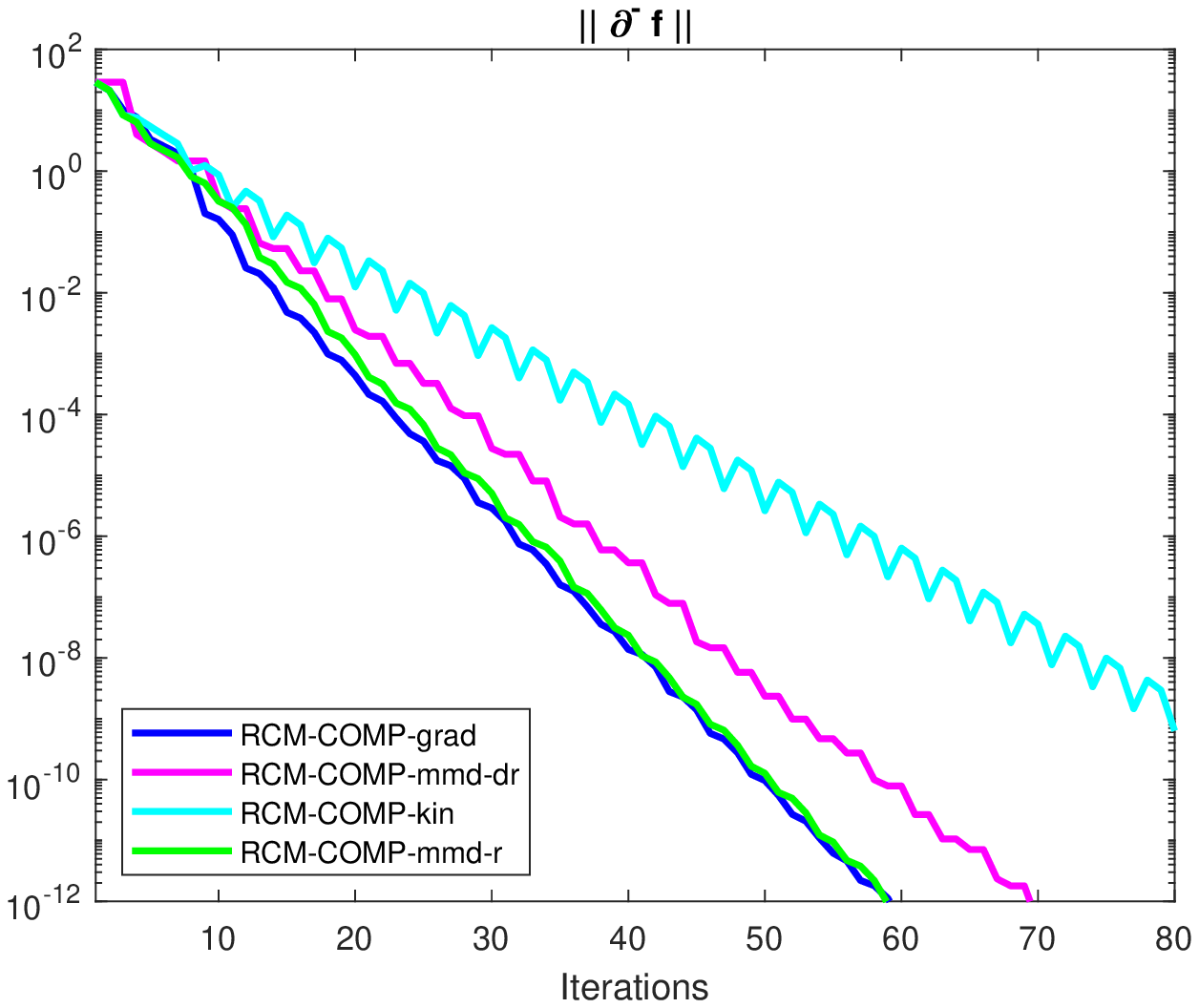}\\
\includegraphics[width=5.9cm]{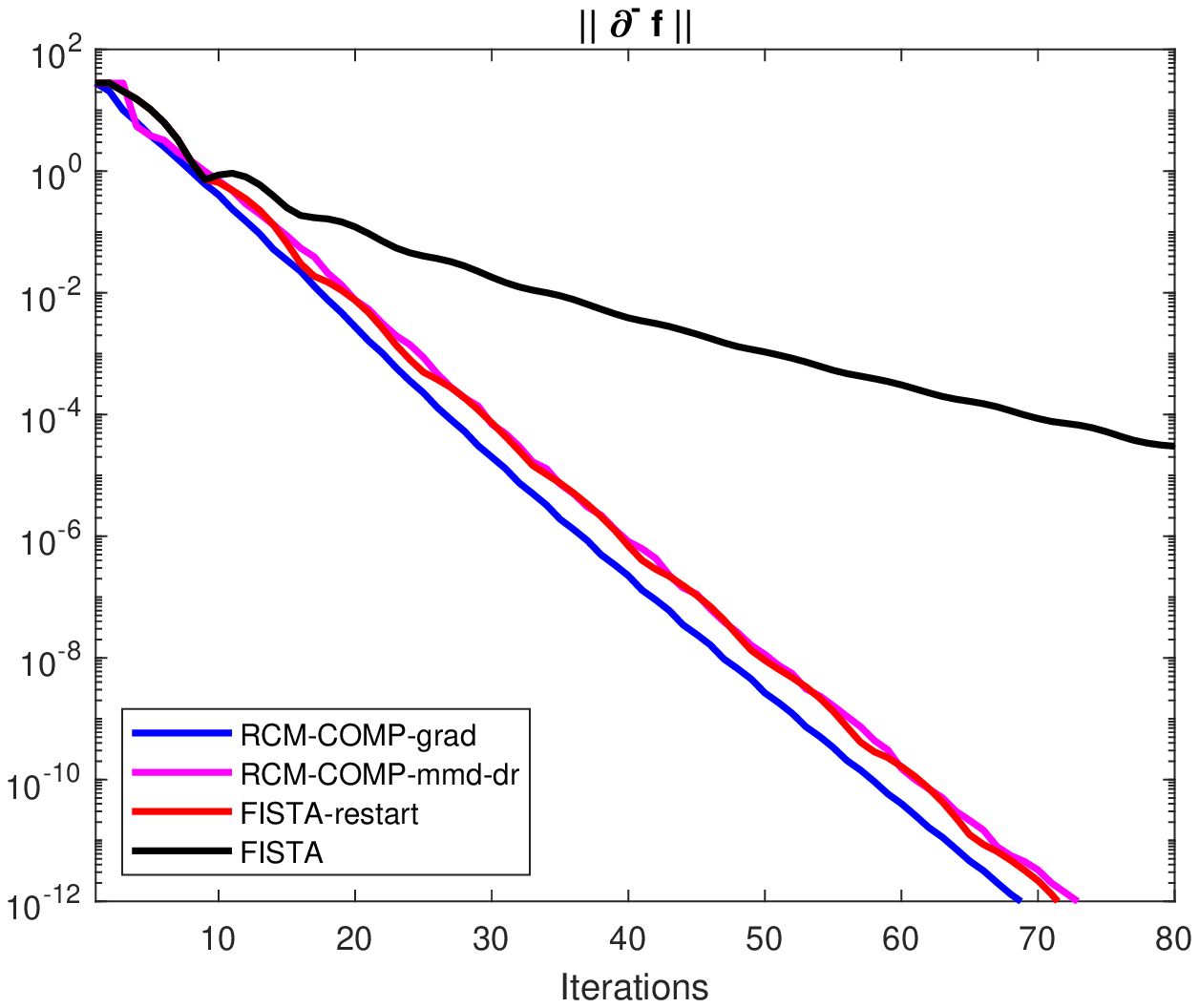}
\includegraphics[width=5.9cm]{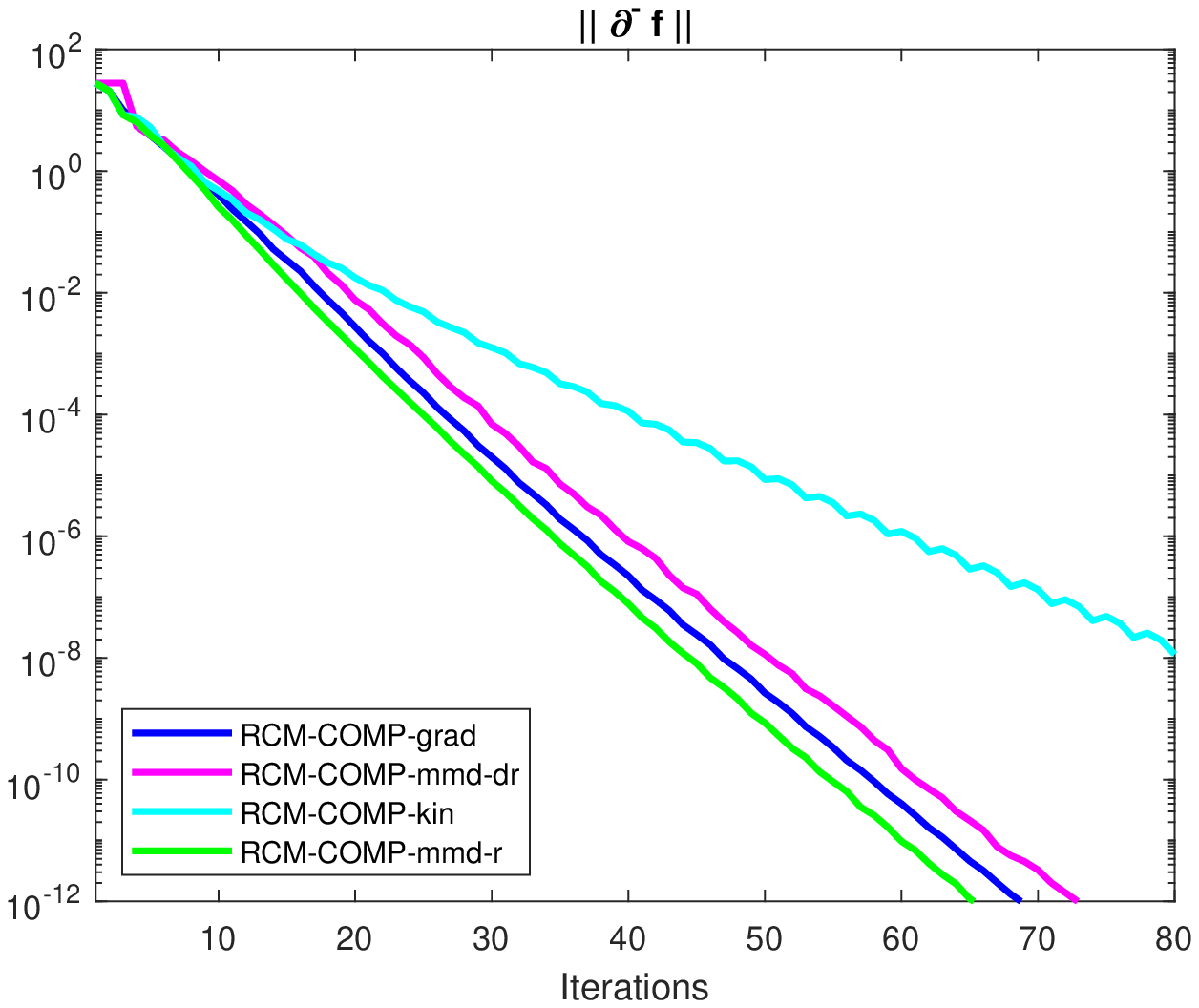}
\caption{Logistic with $\ell^1$-regularization.
At the top we report the result of a single experiment,
at the bottom the average over 100
 repetitions of the experiment.
The plots at left-hand side shows
the decay of $|| \partial^-f ||$
achieved  by RCM-COMP-grad (blue),
RCM-COMP-mmd-dr (magenta), FISTA-restart (red),
and FISTA (black).
At right-hand side we compare the convergence rate
of the  Restart-Conservative method
 with different restart schemes.
RCM-COMP-grad and RCM-COMP-mmd-r (green)
show better convergence rate than the benchmark
FISTA-restart.}
\label{fig:logistic+l1}
\end{figure}
\end{center}
\end{subsection}

\begin{subsection}{LogSumExp with 
$\ell^1$-regularization}
We considered the function $g:\R^n \to \R$ defined
as
\begin{equation*}
g(x) = \rho \log \left(
\sum_{i=1}^m\exp \left( 
\frac{a_i^Tx - b_i }{\rho}\right)  
\right),
\end{equation*}
where $A = (a_1,\ldots,a_m) \in \R^{n\times m}$ and 
$b \in \R^m$ 
 were constructed as 
in Subsection~\ref{subsec:logsumexp}. We set
$\rho =1$.
We studied the function $f:\R^n \to \R$ defined as
\begin{equation}\label{eq:logsumexp+l1}
f(x) =  g(x) + \gamma \sum_{i=1}^n|x_i|.
\end{equation}
We set $\gamma=\frac12 || \nabla g(0) ||_{\infty}$,
in order to guarantee that the minimizer of $f$ 
is not the origin. 
Let ${ L}$ be the Lipschitz constant of the function 
$\nabla g$.
We minimized \eqref{eq:logistic+l1}
 using the following algorithms:
\begin{itemize}
\item FISTA with step-size $s=\frac{1}{{ L}}$;
\item FISTA-restart with step-size $s=\frac{1}{{ L}}$;
\item RCM-COMP-grad with step-size $h=\frac{1}{\sqrt{{ L}}}$;
\item RCM-COMP-mmd-dr with step-size $h=\frac{1}{\sqrt{{ L}}}$;
\item RCM-COMP-mmd-r with step-size $h=\frac{1}{\sqrt{{ L}}}$;
\item RCM-COMP-kin with step-size $h=\frac{1}{\sqrt{{ L}}}$.
\end{itemize}
The results are shown in 
Figure~\ref{fig:logsumexp+l1}.
We measured the convergence rate by considering
the decay of $||\partial^- f||$ along the sequences
generated by the methods.
We observe that 
RCM-COMP-grad
is  the most performing method.
Moreover on average, all the other RCM methods
show performances very close to
the benchmark FISTA-restarted. 
\begin{center}
\begin{figure}
\includegraphics[width=5.9cm]{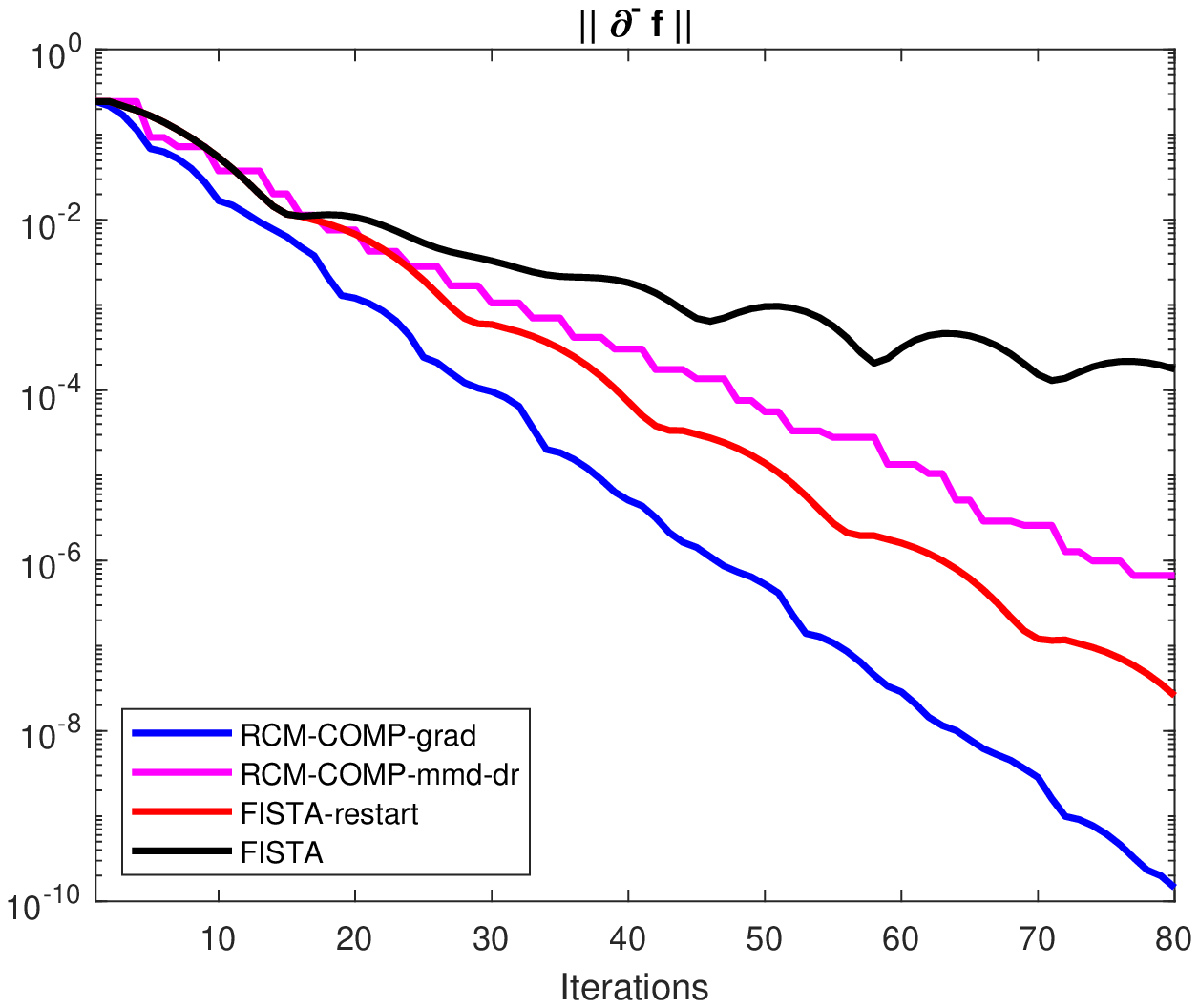}
\includegraphics[width=5.9cm]{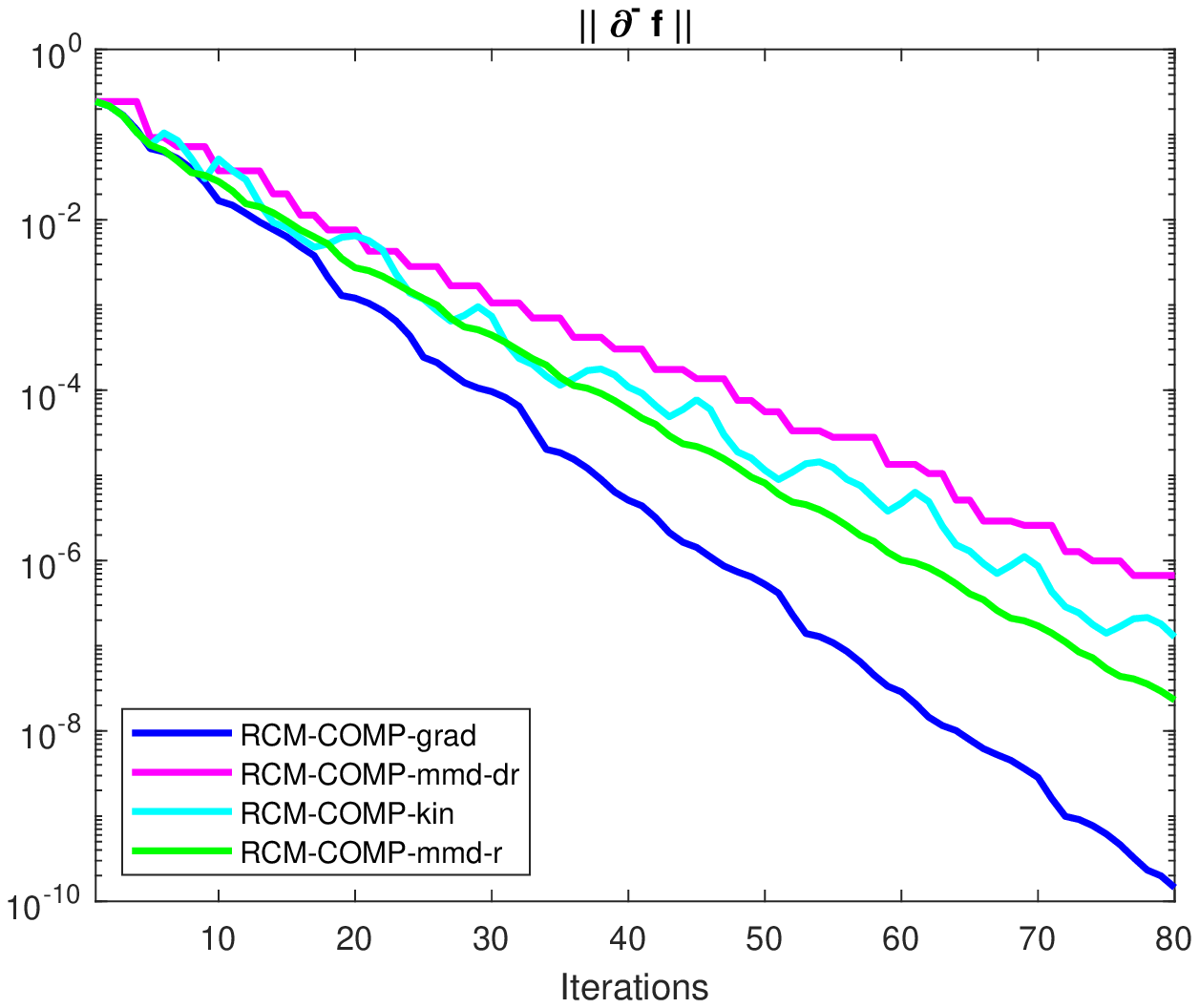}\\
\includegraphics[width=5.9cm]{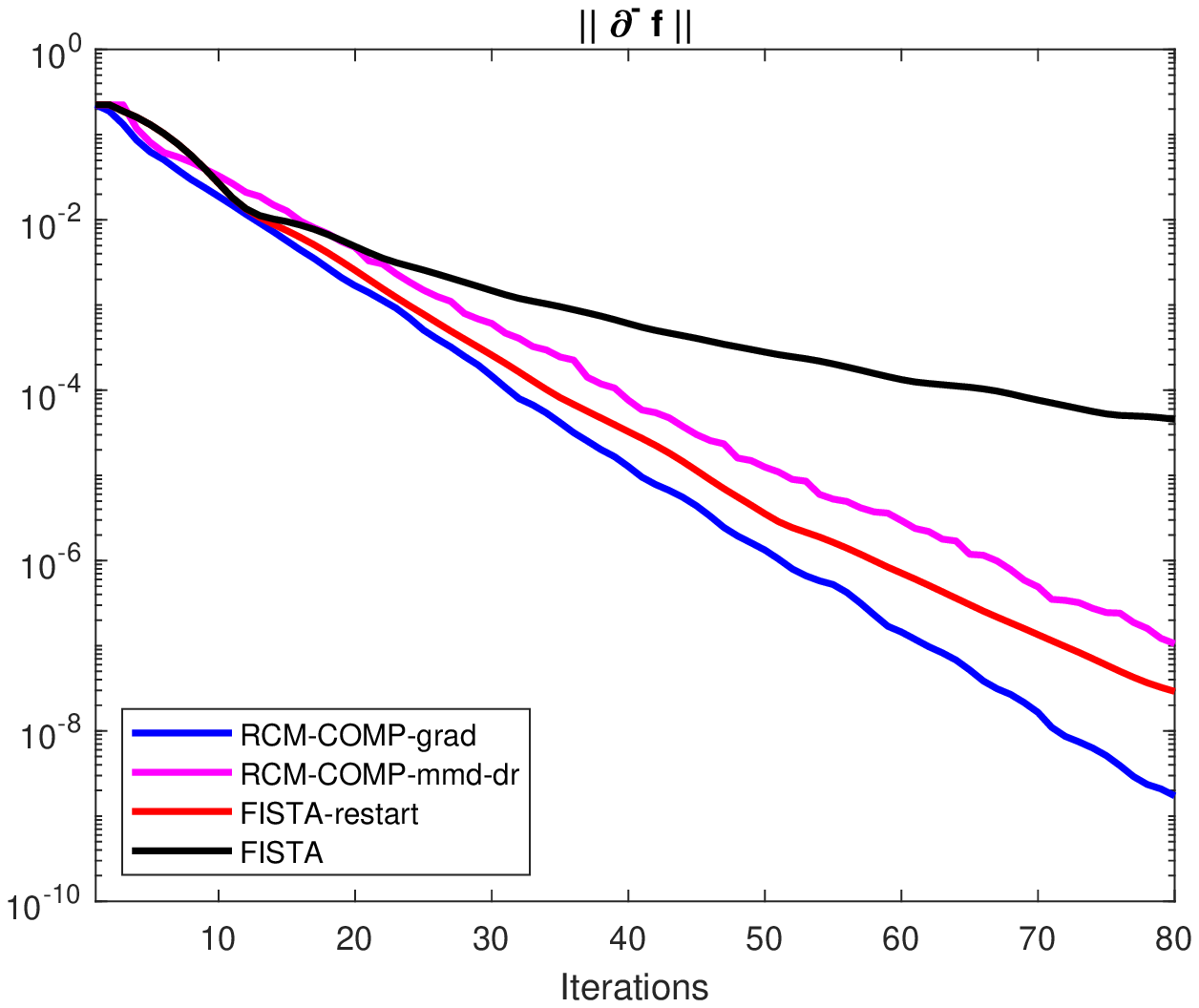}
\includegraphics[width=5.9cm]{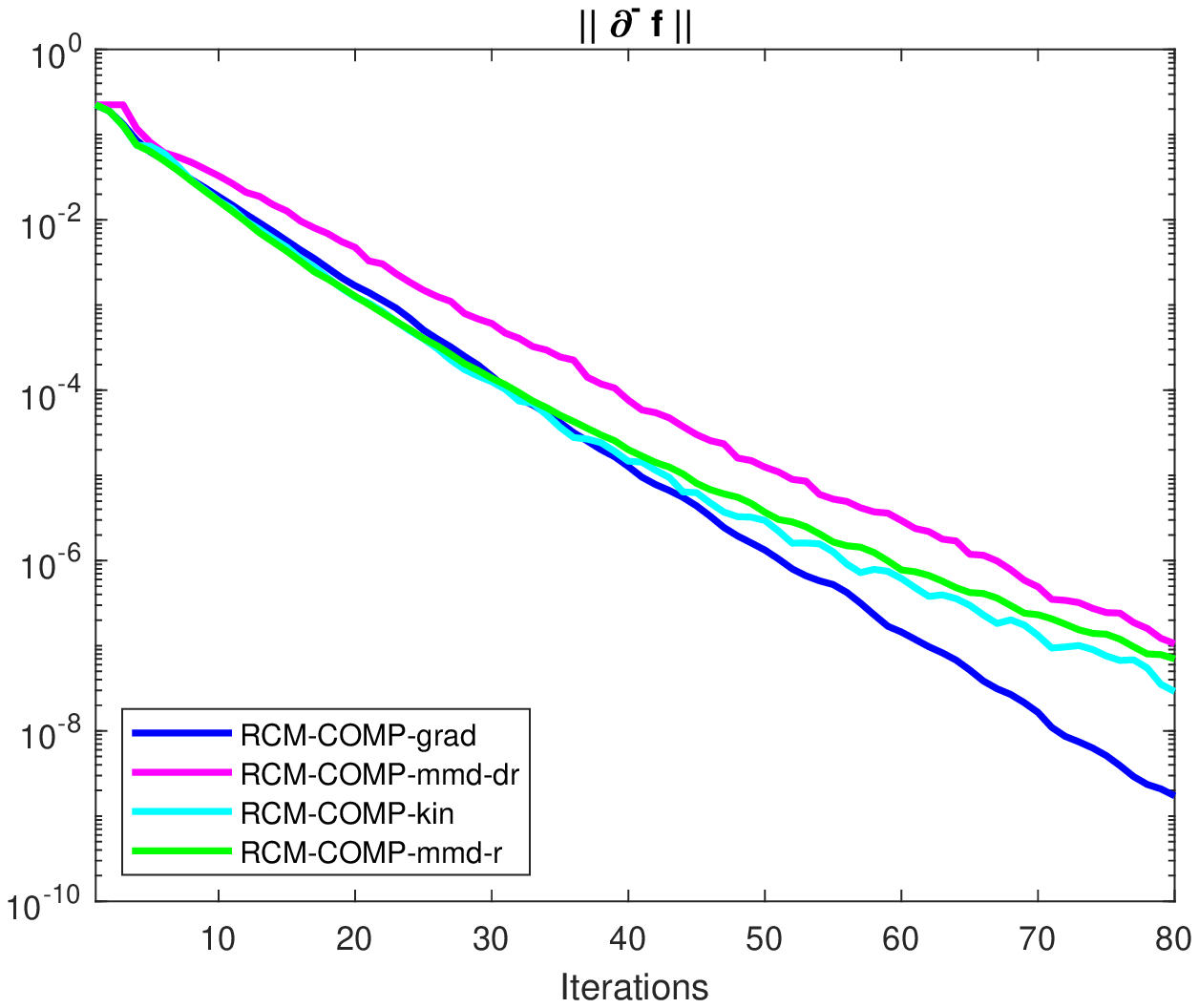}
\caption{LogSumExp with $\ell^1$-regularization.
At the top we report the result of a single experiment,
at the bottom the average over 100
 repetitions of the experiment.
The plots
at left-hand side shows
the decay of $|| \partial^-f ||$
 achieved 
by RCM-COMP-grad (blue), RCM-COMP-mmd-dr (magenta), 
FISTA-restart (red), and
FISTA (black).
At right-hand side we compare the convergence rate
of the  Restart-Conservative method
 with different restart schemes.
 RCM-COMP-grad
is the most performing and it
 shows better convergence rate than the benchmark
FISTA-restart.}
\label{fig:logsumexp+l1}
\end{figure}
\end{center}
\end{subsection}
\end{section}

\begin{section}{Conclusions}

In a series of recent works (see, e.g.,
\cite{SBC1}, \cite{SBC}, 
\cite{A16}, \cite{SJ18}, \cite{A18}, \cite{SJ19}) the
connection between ODEs with suitable
friction term and {discrete
optimization} algorithms with
suitable momentum term has been investigated
from both theoretical and computational point
of view. {Conversely, in the
present work we investigate optimization method 
derived from a conservative ODE (i.e., without 
friction) with suitable restarting
criteria.}

{In the first part of the paper we propose a continuous-time
optimization method for convex functions starting from a conservative ODE coupled with an original restart procedure based on the
mean dissipation of the kinetic energy. Here, our main contribution
consists in proving a linear convergence result.
Indeed, with our stopping
procedure we can prove the 
boundedness 
of the restart time, and therefore we can strengthen
the linear convergence theorem of \cite{TPL},
where the authors assumed the boundedness in the
hypotheses.
As well as in \cite{SBC}, where a restarted 
dissipative dynamics was investigated, our
continuous-time method does not require an estimate
of the strong convexity parameter.
}

In the second part,
a discrete algorithm is derived 
(Restart-Conservative Method, RCM) and
various discrete restart criteria are 
considered{, some of them proposed in
\cite{SI} and \cite{TPL}.
A new contribution is the qualitative global 
convergence result for
RCM with gradient restart (RCM-grad).
To the best of our knowledge, no global convergence
result was available for optimization methods
obtained by the discretization of the conservative
dynamics.}

The numerical tests show that the Restart 
Conservative methods can effectively compete 
with the most performing existing algorithms.
We used as benchmark the restarted versions
of NAG-C and FISTA proposed in \cite{OC}.
{We recall that these methods do
not make use of the constant of strong convexity
of the objective, and they are suitable both
for strongly and non-strongly convex optimization.}
{In the smooth case, in the experiments 
with non-strongly convex functions,
RCM-grad and RCM-mmd-dr have similar 
performances and they both show
a faster convergence rate than NAG-C-restart (see 
Figures~\ref{fig:Logistic},~\ref{fig:LogSumExp}).
In the non-smooth case, when minimizing
a non-strongly convex function with 
$\ell^1$-regularization, the experiments
show that RCM-COMP-grad outperforms FISTA-restart.
Moreover, RCM-COMP-mmd-dr shows 
performances similar to FISTA-restart (see 
Figures~\ref{fig:logistic+l1},~\ref{fig:logsumexp+l1}).}
  
\end{section}

\subsection*{Acknowledgments}
The Authors want to thank
 Prof. G. Savar\'e for the helpful
suggestions and discussions.

\appendix
\section{Proof of Proposition~\ref{prop:1d_finite}}
\label{App_dim_1}
\begin{proof}
Let us prove that the function $t\mapsto E_K(t)$ has
a local maximum in $[0,+\infty)$.
By contradiction, if $t\mapsto E_K(t)$ has
no local maxima, then $t\mapsto E_K(t)$ is 
injective (otherwise we can apply twice Weierstrass 
Theorem and we can find a local maximum).
Since $t\mapsto E_K(t)$ is continuous, it has to
be strictly increasing.
This implies that $t \mapsto \dot x(t)$ can not
change sign and hence that it is monotone as well.
Moreover, it follows that $t \mapsto x(t)$ is
monotone as well.
Since both $x(t)$ and $\dot x(t)$ remain bounded
for every $t \in [0, +\infty)$, there exist 
$x_\infty, v_\infty \in \R$ such that
\[
\lim_{t \to +\infty}x(t) = x_{\infty} \,\,\,\,\,
\mbox{and }\,\, 
\lim_{t \to +\infty}\dot x(t) = v_{\infty}.
\]
On the other hand, $v_{\infty}$ should be zero, and
this is a contradiction.

Let $\bar t$ be a point of local maximum for
the kinetic energy function $t \mapsto E_K(t)$. 
This implies that  $|\dot{x}(\bar t)| > 0$.
The conservation of the total mechanical energy ensures that
the function $t\mapsto f(x(t))$ attains a local minimum at $\bar t$.
Using the Implicit Function Theorem, we obtain
that $t \mapsto x(t)$ is a local homeomorphism 
around $\bar t$.
This implies that $x(\bar t)$ is a point of local
minimum for $f$.   
\end{proof}

\section{Proof of Proposition \ref{prop:est_time_1d}}
\label{App_2dim}
\begin{proof}
Without loss of generality, we can assume that $x^*=0$ and that $x_0>0$.
We define a strongly convex function $g: \R \mapsto \R$ as follows:
\[
g(x) := {\frac12 \mu |x-x^*|^2=}
\frac12 \mu |x|^2.
\]
We claim that, for every $y \in [0,x_0]$, the following inequality is satisfied:
\begin{equation} \label{eq:comp_pot_1d}
f(x_0) - f(y) \geq g(x_0) - g(y).
\end{equation}
Indeed, we have that
{
\begin{align*}
f(x_0) - g(x_0)& = f(y) - g(y) + 
\int_y ^{x_0}(f'(u) - g'(u))\,du
\geq f(y) - g(y),
\end{align*}
since $f'(u) - g'(u)\geq 0$ for every $u\geq 0$.}
Combining \eqref{eq:time_1d} and \eqref{eq:comp_pot_1d} we {obtain}:
\begin{align*}
t_1  =  \int_0^{x_0}\frac{1}{\sqrt{2(f(x_0)-f(y))}}dy 
 \leq & \int_0^{x_0}\frac{1}{\sqrt{2(g(x_0)-g(y))}}dy\\
& = \int_0^{x_0}\frac{1}{\sqrt{\mu(x_0^2- y^2)}}dy 
 = \frac{\pi}{2\sqrt{\mu}}.
\end{align*}
This completes the proof.  
\end{proof}

\section{Proof of Lemma \ref{lem:est_decr} }

\begin{proof}
Up to a linear orthonormal change of coordinates, we can assume that the
function $f$ is of the form
\[
f(x) = \sum_{i=1}^n { \lambda}_i \frac{x_i^2}{2}.
\]
Hence, the differential system \eqref{eq:Cau_prob_nd} becomes
\[
\begin{cases}
\ddot{x}_1 + { \lambda}_1 x_1 = 0, \\
\vdots \\
\ddot{x}_n + { \lambda}_n x_n = 0,
\end{cases}
\]
i.e., the components evolve independently one of each other.
If the Cauchy datum is
\[
x(0) = (x_{1,0}, \, \ldots \, x_{n,0} ) \,\,\, \mbox{and }\,\, \dot{x}(0)=0,
\]
then we can compute the expression of
the kinetic energy function $E_K: t\mapsto \frac12 |\dot{x}(t)|^2$:
\begin{equation*}
E_K(t) = \sum_{i=1}^n { \lambda}_i
\frac{x_{i,0}^2}{2} {\sin^2(
\sqrt{{ \lambda}_i} t)}.
\end{equation*}
For every $0 \leq t\leq \frac{\pi}{2\sqrt{{ \lambda}_n}}$, we have that
\[
0 \leq \sin( \sqrt{{ \lambda}_1} t )\leq \ldots \leq \sin( \sqrt{{ \lambda}_n} t ),
\]
and then we deduce that
\begin{equation*}
E_K(t) \geq \left( \sum_{i=1}^n { \lambda}_i\frac{x_{i,0}^2}{2} \right) {\sin^2(\sqrt{{ \lambda}_1} t)},
\end{equation*}
for every $t \in [ 0, {\pi}/{(2\sqrt{{ \lambda}_n})} ] $. Evaluating the last inequality for
$t = \frac{\pi}{2\sqrt{{ \lambda}_n}}$ and using the conservation of the energy, we obtain the
thesis.    
\end{proof}


\end{document}